\documentclass{amsart}
\usepackage[utf8]{inputenc}
\usepackage[margin=1in,left=1in]{geometry}
\usepackage{setspace}
\usepackage{graphicx}
\usepackage{amssymb, amsmath, amsthm, graphics}
\usepackage{mathabx,epsfig}
\usepackage[mathscr]{euscript}
\usepackage{tikz}
\usepackage{pigpen}
\usepackage{bm}
\usetikzlibrary{calc}
\usetikzlibrary{matrix}
\usepackage{latexsym}
\usepackage[all,knot]{xy}

\usepackage{color}

\usepackage{bbold}

\input xy
\xyoption{all}
\usepackage{euscript}
\usepackage{multirow}
\usepackage{etex, pictexwd,dcpic}
\usetikzlibrary{positioning}
\usetikzlibrary{shapes.geometric}
\usetikzlibrary{shapes.misc}
\usetikzlibrary{calc}
\usetikzlibrary{positioning}

\newcommand{\po}{\ar@{}[dr]|{\text{\pigpenfont R}}}
\newcommand{\pb}{\ar@{}[dr]|{\text{\pigpenfont J}}}

\usepackage{pgflibraryarrows}
\usepackage{pgflibrarysnakes}

\usetikzlibrary{trees} 
\usetikzlibrary[trees] 

\newtheorem{theorem}{Theorem}
\theoremstyle{definition}
\newtheorem{definition}{Definition}[section]
\newtheorem{notation}{Notation}[section]
\newtheorem{lemma}{Lemma}[section]
\newtheorem{corollary}{Corollary}[section]
\newtheorem{proposition}{Proposition}[section]

\newtheorem{remark}{Remark}[section]

\newtheorem{example}{Example}[section]

\numberwithin{equation}{section}

\def\iunknot{
{\xy
(0,-5);(0,-5)**\crv{(5,-5)&(5,5)&(-5,5)&(-5,-5)}
?>(.3)*\dir{>}
,(-4,-5)*{\scriptstyle{1}}
\endxy}
}

\def\ev{
{\xy
(-5,-5);(5,-5)**\crv{(-5,5)&(0,5)&(5,5)}
?>(1)*\dir{>}
,(-4,-5)*{\scriptstyle{1}},(6,-5)*{\scriptstyle{1}}
\endxy}
}

\def\evop{
{\xy
(-5,-5);(5,-5)**\crv{(-5,5)&(0,5)&(5,5)}
?>(0)*\dir{<}
,(-4,-5)*{\scriptstyle{1}},(6,-5)*{\scriptstyle{1}}
\endxy}
}

\def\coev{
{\xy
(-5,5);(5,5)**\crv{(-5,-5)&(0,-5)&(5,-5)}
?>(0)*\dir{<}
,(-4,5)*{\scriptstyle{1}},(6,5)*{\scriptstyle{1}}
\endxy}
}

\def\coevop{
{\xy
(-5,5);(5,5)**\crv{(-5,-5)&(0,-5)&(5,-5)}
?>(1)*\dir{>}
,(-4,5)*{\scriptstyle{1}},(6,5)*{\scriptstyle{1}}
\endxy}
}

\def\esse{
{\xy
(-5,-10);(-5,10)**\crv{(-1,-6)&(0,-5)&(0,-2)&(0,2)&(0,5)&(-1,6)}
?>(.55)*\dir{>}?>(.1)*\dir{>}?>(.95)*\dir{>}
,(5,-10);(5,10)**\crv{(1,-6)&(0,-5)&(0,-2)&(0,2)&(0,5)&(1,6)}
?>(.1)*\dir{>}?>(.95)*\dir{>}
,(-7,-10)*{\scriptstyle{1}}
,(7,-10)*{\scriptstyle{1}}
,(-7,10)*{\scriptstyle{1}}
,(7,10)*{\scriptstyle{1}}
,(2,0)*{\scriptstyle{2}}
\endxy}\,
}

\def\ti{
{\xy
(-5,-10);(-5,10)**\crv{(-1,-6)&(0,-5)&(0,-2)&(0,2)&(0,5)&(-1,6)}
?>(.55)*\dir{>}?>(.1)*\dir{>}?>(.95)*\dir{>}
,(5,-10);(5,10)**\crv{(1,-6)&(0,-5)&(0,-2)&(0,2)&(0,5)&(1,6)}
?>(.1)*\dir{>}?>(.95)*\dir{>}
,(0,-10);(0,10)**\dir{-}?>(.15)*\dir{>}?>(.95)*\dir{>}
,(-7,-10)*{\scriptstyle{1}}
,(7,-10)*{\scriptstyle{1}}
,(-7,10)*{\scriptstyle{1}}
,(7,10)*{\scriptstyle{1}}
,(-1,-11)*{\scriptstyle{1}}
,(-1,11)*{\scriptstyle{1}}
,(2,0)*{\scriptstyle{3}}
\endxy}
}

\def\essecircesse{
{\xy
(0,-10);(0,10)**\crv{(0,-7)&(0,-6)&(0,-5)&(-4,0)&(0,5)&(0,6)&(0,7)}
?>(.1)*\dir{>}?>(.99)*\dir{>}?>(.65)*\dir{>}
,(0,-10);(0,10)**\crv{(0,-7)&(0,-6)&(0,-5)&(4,0)&(0,5)&(0,6)&(0,7)}
?>(.65)*\dir{>}
,(0,10);(-5,14)**\crv{(0,10)&(0,11)&(-1,12)}?>(.95)*\dir{>}
,(0,10);(5,14)**\crv{(0,10)&(0,11)&(1,12)}?>(.95)*\dir{>}
,(0,-10);(-5,-14)**\crv{(0,-10)&(0,-11)&(-1,-12)}?>(.85)*\dir{<}
,(0,-10);(5,-14)**\crv{(0,-10)&(0,-11)&(1,-12)}?>(.85)*\dir{<}
,(-7,-15)*{\scriptstyle{1}}
,(7,-15)*{\scriptstyle{1}}
,(-7,15)*{\scriptstyle{1}}
,(7,15)*{\scriptstyle{1}}
,(-4,0)*{\scriptstyle{1}}
,(4,0)*{\scriptstyle{1}}
,(1.5,-9)*{\scriptstyle{2}}
,(1.5,9)*{\scriptstyle{2}}
\endxy}
}

\def\idone{
{\xy
(0,-10);(0,10)**\dir{-}?>(.5)*\dir{>}
,(-2,-10)*{\scriptstyle{1}}
,(-2,10)*{\scriptstyle{1}}
\endxy}
}

\def\idoner{
{\xy
(0,-10);(0,10)**\dir{-}?>(.5)*\dir{>}
,(2,-10)*{\scriptstyle{1}}
,(2,10)*{\scriptstyle{1}}
\endxy}
}

\def\idtwo{
{\xy
(-5,-10);(-5,10)**\dir{-}?>(.5)*\dir{>}
,(5,-10);(5,10)**\dir{-}?>(.5)*\dir{>}
,(-7,-10)*{\scriptstyle{1}}
,(7,-10)*{\scriptstyle{1}}
,(-7,10)*{\scriptstyle{1}}
,(7,10)*{\scriptstyle{1}}
\endxy}
}

\def\splus{\xy 
(4,-2.5);(-4,2.5)**\crv{(4,-2.25)&(4,-2.1)&(0,0)&(-4,2.1)&(-4,2.25)}
?>(.7)*\dir{>}?>(.5)*{\color{white}\bullet},
(-4,-2.5);(4,2.5)**\crv{(-4,-2.25)&(-4,-2.1)&(0,0)&(4,2.1) & (4,2.25)}
?>(.7)*\dir{>}
\endxy}

\newcommand{\splusk}[1]{\left.{\raisebox{-3.5pt}{\xy
(4,-7.5);(-4,-2.5)**\crv{(4,-7.25)&(4,-7.1)&(0,-5)&(-4,-2.9)& (-4,-2.75)}
?>(.5)*{\color{white}\bullet},
,(-4,-7.5);(4,-2.5)**\crv{(-4,-7.25)&(-4,-7.1)&(0,-5)&(4,-2.9)& (4,-2.75)}
,(4,-2.5);(-4,2.5)**\crv{(4,-2.25)&(4,-2.1)&(0,0)&(-4,2.1)& (-4,2.25)}
?>(.7)*\dir{>}?>(.5)*{\color{white}\bullet},
(-4,-2.5);(4,2.5)**\crv{(-4,-2.25)&(-4,-2.1)&(0,0)&(4,2.1) &(4,2.25)}
?>(.7)*\dir{>}
,(4,7);(-4,12)**\crv{(4,7.25)&(4,7.35)&(0,9.5)&(-4,11.6)& (-4,11.75)}
?>(.7)*\dir{>}?>(.5)*{\color{white}\bullet},
,(-4,7);(4,12)**\crv{(-4,7.25)&(-4,7.35)&(0,9.5)&(4,11.6)& (4,11.75)}
?>(.7)*\dir{>}
,(-4,5.5)*{\vdots}
,(4,5.5)*{\vdots}
\endxy}}\,\,\right\}\text{\small{${#1}$ \rm crossings}}}

\newcommand{\trsplusk}[1]{\left.{\raisebox{-3.5pt}{\xy
(4,-7.5);(-4,-2.5)**\crv{(0,-5)&(-4,-2.9)& (-4,-2.75)}
?>(.2)*{\color{white}\bullet},
,(-4,-7.5);(4,-2.5)**\crv{(0,-5)&(4,-2.9)& (4,-2.75)}
,(4,-2.5);(-4,2.5)**\crv{(4,-2.25)&(4,-2.1)&(0,0)&(-4,2.1)& (-4,2.25)}
?>(.7)*\dir{>}?>(.5)*{\color{white}\bullet},
(-4,-2.5);(4,2.5)**\crv{(-4,-2.25)&(-4,-2.1)&(0,0)&(4,2.1) &(4,2.25)}
?>(.7)*\dir{>}
,(4,7);(-4,12)**\crv{(4,7.25)&(4,7.35)&(0,9.5)}
?>(.7)*\dir{>}?>(.79)*{\color{white}\bullet},
,(-4,7);(4,12)**\crv{(-4,7.25)&(-4,7.35)&(0,9.5)}
?>(.7)*\dir{>}
,(-4,5.5)*{\vdots}
,(4,5.5)*{\vdots}
,(-4,12);(-4,-7.5)**\crv{(-5,12.4)&(-6,12.8)&(-9,1.25)&(-6,-8.3)&(-5,-7.9)}
,(4,12);(4,-7.5)**\crv{(5,12.4)&(6,12.8)&(9,1.25)&(6,-8.3)&(5,-7.9)}
\endxy}}\,\,\right\}\text{\small{${#1}$ \rm crossings}}}

\newcommand{\hopf}{{\raisebox{7pt}{\xy
(4,-7.5);(-4,-2.5)**\crv{(0,-5)&(-4,-2.9)& (-4,-2.75)}
?>(.2)*{\color{white}\bullet},
,(-4,-7.5);(4,-2.5)**\crv{(0,-5)&(4,-2.9)& (4,-2.75)}
,(4,-2.5);(-4,2.5)**\crv{(4,-2.25)&(4,-2.1)&(0,0)}
?>(.7)*\dir{>}?>(.79)*{\color{white}\bullet},
(-4,-2.5);(4,2.5)**\crv{(-4,-2.25)&(-4,-2.1)&(0,0)}
?>(.7)*\dir{>}
,(-4,2.5);(-4,-7.5)**\crv{(-5,2.9)&(-6,3.3)&(-9,-2.5)&(-6,-8.3)&(-5,-7.9)}
,(4,2.5);(4,-7.5)**\crv{(5,2.9)&(6,3.3)&(9,-2.5)&(6,-8.3)&(5,-7.9)}
\endxy}}}

\newcommand{\trefoil}{{\raisebox{1pt}{\xy
(4,-7.5);(-4,-2.5)**\crv{(0,-5)&(-4,-2.9)& (-4,-2.75)}
?>(.2)*{\color{white}\bullet},
,(-4,-7.5);(4,-2.5)**\crv{(0,-5)&(4,-2.9)& (4,-2.75)}
,(4,-2.5);(-4,2.5)**\crv{(4,-2.25)&(4,-2.1)&(0,0)&(-4,2.1)& (-4,2.25)}
?>(.7)*\dir{>}?>(.5)*{\color{white}\bullet},
(-4,-2.5);(4,2.5)**\crv{(-4,-2.25)&(-4,-2.1)&(0,0)&(4,2.1) &(4,2.25)}
?>(.7)*\dir{>}
,(4,2.5);(-4,7.5)**\crv{(4,2.75)&(4,2.85)&(0,5)}
?>(.79)*{\color{white}\bullet},
,(-4,2.5);(4,7.5)**\crv{(-4,2.75)&(-4,2.85)&(0,5)}
,(-4,7.5);(-4,-7.5)**\crv{(-5,7.9)&(-6,8.3)&(-9,0)&(-6,-8.3)&(-5,-7.9)}
,(4,7.5);(4,-7.5)**\crv{(5,7.9)&(6,8.3)&(9,0)&(6,-8.3)&(5,-7.9)}
\endxy}}}

\def\tresse{
{\xy
(-5,-10);(-5,10)**\crv{(-1,-6)&(0,-5)&(0,-2)&(0,2)&(0,5)&(-1,6)}
?>(.55)*\dir{>}?>(.1)*\dir{>}?>(.95)*\dir{>}
,(5,-10);(5,10)**\crv{(1,-6)&(0,-5)&(0,-2)&(0,2)&(0,5)&(1,6)}
?>(.1)*\dir{>}?>(.95)*\dir{>}
,(-5,10);(-5,-10)**\crv{(-6,10.5)&(-7,10.8)&(-10,0)&(-7,-10.8)&(-6,-10.5)}
,(5,10);(5,-10)**\crv{(6,10.5)&(7,10.8)&(10,0)&(7,-10.8)&(6,-10.5)}
,(-8,0)*{\scriptstyle{1}}
,(8,0)*{\scriptstyle{1}}
,(2,0)*{\scriptstyle{2}}
\endxy}\,
}

\def\twocircles{
{\xy
(0,-10);(0,-10)**\crv{(5,-10)&(5,10)&(-5,10)&(-5,-10)}
?>(.3)*\dir{>}
,(11,-10);(11,-10)**\crv{(16,-10)&(16,10)&(6,10)&(6,-10)}
?>(0.7)*\dir{<}
,(7.2,-9)*{\scriptstyle{1}}
,(4,-9)*{\scriptstyle{1}}
\endxy}}

\def\zorro{
{\xy
(-3,0);(-3,10)**\crv{(-1,2)&(0,3)&(0,4)&(0,6)&(0,7)&(-1,8)}
?>(.55)*\dir{>}
?>(.95)*\dir{>}
,(3,0);(3,10)**\crv{(1,2)&(0,3)&(0,4)&(0,6)&(0,7)&(1,8)}
?>(.95)*\dir{>}
,(3,-10);(3,0)**\crv{(5,-8)&(6,-7)&(6,-6)&(6,-4)&(6,-3)&(5,-2)}
?>(.55)*\dir{>}?>(.1)*\dir{>}?>(.99)*\dir{>}
,(9,-10);(9,0)**\crv{(7,-8)&(6,-7)&(6,-6)&(6,-4)&(6,-3)&(7,-2)}
?>(.1)*\dir{>}
,(9,0);(10,10)**\crv{(10,1)&(10,3)&(10,4)&(10,6)&(10,7)&(10,8)}
?>(.5)*\dir{>}
,(-3,0);(-4,-10)**\crv{(-4,-1)&(-4,-3)&(-4,-4)&(-4,-6)&(-4,-7)&(-4,-8)}
?>(.5)*\dir{<}
,(-5,-10)*{\scriptstyle{1}}
,(10,-10)*{\scriptstyle{1}}
,(2,-10)*{\scriptstyle{1}}
,(4,10)*{\scriptstyle{1}}
,(-4,10)*{\scriptstyle{1}}
,(11,10)*{\scriptstyle{1}}
,(1,0)*{\scriptstyle{1}}
,(4.5,-5)*{\scriptstyle{2}}
,(1.5,5)*{\scriptstyle{2}}
\endxy}\,
}

\def\coso{
{\xy
(0,-10);(0,10)**\crv{(0,-7)&(0,-6)&(0,-5)&(-4,0)&(0,5)&(0,6)&(0,7)}
?>(.1)*\dir{>}?>(.99)*\dir{>}?>(.65)*\dir{>}
,(0,-10);(0,10)**\crv{(0,-7)&(0,-6)&(0,-5)&(4,0)&(0,5)&(0,6)&(0,7)}
,(0,10);(-5,14)**\crv{(0,10)&(0,11)&(-1,12)}?>(.95)*\dir{>}
,(0,10);(5,14)**\crv{(0,10)&(0,11)&(1,12)}?>(.95)*\dir{>}
,(0,-10);(-5,-14)**\crv{(0,-10)&(0,-11)&(-1,-12)}?>(.85)*\dir{<}
,(0,-10);(5,-14)**\crv{(0,-10)&(0,-11)&(1,-12)}?>(.85)*\dir{<}
,(10,-14);(2,-2.2)**\crv{(10,-7)&(10,-6)&(4,-4)&(1,-4)}
?>(.1)*\dir{>}?>(.65)*\dir{>}
,(10,14);(2,2.2)**\crv{(10,7)&(10,6)&(4,4)&(1,4)}
?>(.1)*\dir{<}?>(.65)*\dir{<}
,(-7,-15)*{\scriptstyle{1}}
,(7,-15)*{\scriptstyle{1}}
,(-7,15)*{\scriptstyle{1}}
,(7,15)*{\scriptstyle{1}}
,(-4,0)*{\scriptstyle{1}}
,(4,0)*{\scriptstyle{2}}
,(1.5,-9)*{\scriptstyle{2}}
,(1.5,9)*{\scriptstyle{2}}
,(11,-15)*{\scriptstyle{1}}
,(11,15)*{\scriptstyle{1}}
,(1.5,-5)*{\scriptstyle{1}}
,(1.5,5)*{\scriptstyle{1}}
\endxy}
}

\def\idtwob{\mathrm{id}_2}
\def\esseb{S}
\def\essecircesseb{S\circ S}

\title{Computing the Khovanov homology of 2 strand braid links via generators and relations}

\author{Domenico Fiorenza}
\address{Sapienza Universit\`a di Roma; Dipartimento di Matematica ``Guido Castelnuovo'', P.le Aldo Moro, 5 - 00185 - Rome, Italy; 
}
\email{fiorenza@mat.uniroma1.it}
\author{Omid Hurson}
\address{Formerly at PhD school in Mathematics, Faculty of Mathematics, Universit\"at Wien, Oskar-Morgenstern-Platz 1 -
1090 - Vienna, Austria}
\email{hurson\_omid@yahoo.com}

\begin{document}

\maketitle

\begin{abstract}
In {\it Homfly polynomial via an invariant of colored plane graphs},
Murakami, Ohtsuki, and Yamada provide a
state-sum description of the level $n$ Jones
polynomial of an oriented link in terms of a suitable braided
monoidal category whose morphisms are $\mathbb{Q}[q,q^{-1}]$-linear combinations of
oriented trivalent planar graphs, and give a corresponding
description for the HOMFLY-PT polynomial. 
We
extend this construction and express the Khovanov--Rozansky
homology of an oriented link in terms of a combinatorially defined
category whose morphisms are equivalence classes of formal
complexes of (formal direct sums of shifted) oriented trivalent
plane graphs. By working combinatorially, one avoids many of the
computational difficulties involved in the matrix factorization
computations of the original Khovanov--Rozansky formulation: one
systematically uses combinatorial relations satisfied by these
matrix factorizations to simplify the computation at a level that is
easily handled. By using this technique, we are able to provide a
computation of the level $n$ Khovanov--Rozansky invariant of the 2-strand braid link with $k$ crossings, for arbitrary $n$ and $k$, confirming
and extending previous results and conjectural predictions by
Anokhina--Morozov, Beliakova--Putyra--Wehrli, Carqueville--Murfet, Dolotin--Morozov, Gukov--Iqbal--Kozcaz--Vafa, Nizami--Munir--Sohail--Usman, and Rasmussen.
\end{abstract}

\setcounter{tocdepth}{1}
\tableofcontents

\section{Introduction}

A nontrivial refinement of the Jones polynomial has been obtained by Khovanov and Rozansky \cite{Khovanov-Rozansky}. Their link invariant is a polynomial in \emph{two} variables $q,t$ and their inverses, which specializes to the Jones polynomial when it is evaluated at $t=-1$. For instance
\[
KR_n\left(\,\hopf\,\right)=q^{2(n-1)}\left(q^{1-n}+t^2q^3[n-1]_q\right)[n]_q.
\]
The Khovanov-Rozansky invariant is generally very hard to compute, as it does not satisfy a skein relation, and actually
 the original definition of the Khovanov-Rozansky invariant is rather complicated: it involves complexes of matrix factorizations of certain Landau-Ginzburg potentials (polynomials in some set of variables with no constant or linear term, with finite dimensional Jacobian algebra). On the other hand,  Murakami,  Ohtsuki and Yamada have shown in \cite{moy} that, in order to derive the Jones polynomial, the braided monoidal category $\mathrm{Rep}(U_q(\mathfrak{sl}_n))$ appearing in the ``quantum traces'' derivation of the Jones polynomial can be conveniently replaced by a simple purely combinatorial category defined in terms of planar trivalent graphs. This naturally suggests the following question: is it possible to replace also the rather complicated algebraic category of Landau-Ginzburg potentials and matrix factorizations with a simple purely combinatorial category defined in terms of planar trivalent graphs?

\bigskip

The main result of this article consists in giving a positive answer to this question. In Section \ref{chapter:KR} we define a braided monoidal category $\mathbf{KR}$ of formal complexes of (formal direct sums of shifted) MOY-type graphs, that encodes combinatorially all of the features of the category of complexes of matrix factorizations that are used by Khovanov and Rozansky in order to construct their link invariant. More in detail, in \cite{Khovanov-Rozansky}, Khovanov and Rozansky consider a monoidal category which could be called the ``homotopy category of chain complexes in the Landau-Ginzburg category'' and therefore denoted by the symbol $\mathcal{K}(\mathcal{LG})$. Objects of $\mathcal{K}(\mathcal{LG})$ are Landau-Ginzburg potentials, i.e., polynomials $W(\vec{x})$ in the variables $\vec{x}=(x_1,x_2,\dots)$ over the field $\mathbb{Q}$ (or an extension of it), such that the Jacobian ring $\mathbb{Q}[x_1,x_2,\dots]/(\partial_1W,\partial_2W,\dots)$ is finite dimensional. Morphisms between $W_1(\vec{x})$ and $W_2(\vec{y})$ are homotopy equivalence classes of complexes of matrix factorizations of $W_1(\vec{x})-W_2(\vec{y})$. The category $\mathcal{K}(\mathcal{LG})$ is rigid, and it is expected to be braided.\footnote{Unfortunately, we are not aware of any reference proving (or disproving) this fact. Our impression is that most experts are quite certain that $\mathcal{K}(\mathcal{LG})$ is indeed a braided monoidal category.}  Even without knowing for sure whether $\mathcal{K}(\mathcal{LG})$ is braided or not, Khovanov and Rozansky are nevertheless able to define a rigid monoidal functor from the category of tangle diagrams to $\mathcal{K}(\mathcal{LG})$,
\[
KH_n\colon\mathbf{TD}\to \mathcal{K}(\mathcal{LG}),
\]
with $KH_n(\uparrow)=x^{n+1}$, thus realizing $W(x)=x^{n+1}$ as a \emph{braided object} in $\mathcal{K}(\mathcal{LG})$. For a link $\Gamma$ the complex of matrix factorizations $KH_n(\Gamma)$ is a complex of matrix factorizations of the potential 0, and so it is simply a bicomplex of graded $\mathbb{Q}$-vector spaces. Up to homotopy, this is identified by the cohomology of its total complex. This is a finite dimensional  bigraded $\mathbb{Q}$-vector space, called the Khovanov-Rozansky cohomology of the link $\Gamma$. As the only invariant of a finite dimensional bigraded vector space is its graded dimension, the datum of the Khovanov-Rozansky cohomology of $\Gamma$ is equivalently the datum of a polynomial in $\mathbb{Z}[q,q^{-1},t,t^{-1}]$. This is precisely the Khovanov-Rozansky invariant mentioned above: $KR_n(\Gamma)=\dim_{q,t}KH_n(\Gamma)$.
\par 
As one could expect, the construction of $KH$ is in two steps: first, Khovanov and Rozansky define a rigid monoidal functor $KH_n\colon \mathbf{PTD}\to \mathcal{K}(\mathcal{LG})$ with $KH_n(\uparrow)=x^{n+1}$, where $\mathbf{PTD}$ denotes the rigid monoidal category of plane tangle diangrams modulo plane homotopies, and then they show that $KH$ is invariant with respect to the Reidemeister moves thus inducing a braided monoidal functor 
$KH_n\colon \mathbf{TD}\to \mathcal{K}(\mathcal{LG})$. The proof of this fact is quite nontrivial and  uses subtle properties of the complexes of matrix factorizations involved with the definition of $KH_n$. What we do in Section \ref{chapter:KR} is precisely to extract from Khovanov and Rozansky's proof all the properties of these complexes of matrix factorizations that are actually used in the proof, as well as additional properties demonstrated by Rasmussen in \cite{Rasmussen}, and change them into defining properties for a suitable monoidal category $\mathbf{KR}$. By this, one tautologically has that the monoidal functor $KH$ factors as
 \begin{equation}\label{eq:krr1}\tag{$\ast$}
\xymatrix{
\mathbf{TD}\ar[r]^{Kh_n}\ar@/_2pc/[rr]_{KH_n} &\mathbf{KR} \ar[r]^{\mathrm{mf}} &\mathcal{K}(\mathcal{LG})
}
\end{equation}
for some monoidal functor $\mathrm{mf}\colon \mathbf{KR}\to \mathcal{K}(\mathcal{LG})$ with $\mathrm{mf}(\uparrow)=x^{n+1}$, where `mf' stands for `matrix factorizations'. Moreover $\mathrm{mf}$ is compatible with shifts and direct sums. This implies that it commutes with the tensor product by graded $\mathbb{Q}$-vector spaces (with a fixed basis).
\par
Unfortunately we are not able to prove whether $\mathrm{mf}$ is an embedding of categories. Clearly, it is injective on objects, but its faithfulness on morphisms is elusive: there could be nonequivalent formal complexes of MOY-type graphs leading to homotopy equivalent complexes of matrix factorizations. Actually, our opinion is that $\mathrm{mf}$ should not be expeted to be faithful: it is likely that in $\mathcal{K}(\mathcal{LG})$ there are many more relations than those involved in the proof of the Reidemeister invariance of $KH_n$ by Khovanov and Rozansky and those evidentiated by Rasmussen.
Yet, if for a link $\Gamma$ we have
\[
Kh_n(\Gamma)=[V_\bullet]\otimes \emptyset,
\]
where $\emptyset$ is the empty diagram and  
\[
V_\bullet=\left(\cdots\to V_{i-1}\xrightarrow{0} V_i\xrightarrow{0} V_{i+1}\cdots \right),
\]
is some complex of graded $\mathbb{Q}$-vector spaces with trivial differentials,
then the commutativity of the diagram \ref{eq:krr1} implies that we have
\[
KH_n(\Gamma)=\oplus_i V_i,
\]
with the graded $\mathbb{Q}$-vector space $V_i$ in ``horizontal'' degree $i$. This precisely means that, for those links such that $Kh_n(\Gamma)=[V_\bullet]\otimes \emptyset$, the functor $Kh_n$ computes the Khovanov-Rozansky homology. This fact is extensively used in Section \ref{chapter:computation} to compute the level $n$ Khovanov-Rozansky polynomial for all 2-strands braid links with $k$ crossings (this is actually a knot when $k$ is odd and a 2-component link when $k$ is even). 
Namely, we obtain
\[
Kh_n\left(\,\trsplusk{2k+1}\,\right)=q^{(n-1)(2k+1)}\left(q^{1-n}[n]_q+(q^{-n}+tq^n)\frac{t^2q^4-t^{2k+2}q^{4k+4}}{1-t^2q^4}[n-1]_q \right)
\]
and
\[
Kh_n\left(\,\trsplusk{2k}\,\right)
=q^{(n-1)(2k)}\left(q^{1-n}[n]_q+(q^{-n}+tq^n)\frac{t^2q^4-t^{2k}q^{4k}}{1-t^2q^4}[n-1]_q+t^{2k}q^{4k-1}[n]_q[n-1]_q \right).
\]
 This perfectly matches (and so confirms and extends) the results obtained or predicted for this kind of knots and links by:
  
  \begin{itemize}
  \item Rasmussen for the odd $k$'s and $n\geq 4$ \cite{Rasmussen}; 
  
  \item Carqueville and Murfet for low values of $k$ and $n$ (computer assisted computation) \cite{Carqueville-Murfet};
  
   \item Beliakova--Putyra--Wehrli and Nizami--Munir--Sohail--Usman for $n=2$ and all $k$  \cite{bpw19,nmsu16}; 
  
  \item Gukov--Iqbal--Koz\c{c}az--Vafa for $k=2$ and all $n$ (string theory based prediction) \cite{Gukov};
  
  \item  Anokhina, Dolotin, and Morozov for all $k$ and $n$ (algebraic combinatorics based prediction) \cite{Morozov1,Morozov2}.

\end{itemize}
\vskip .5 cm

\noindent {\it Acknowledgements.} The authors thank Nils Carqueville for several inspiring conversations on the subject of this article. OH thanks Nils Carqueville for constant encouragement and support during his years as a PhD student in Mathematics at University of Vienna. This paper is part of the activities of the MIUR Excellence Department
Projects CUP:B83C23001390001.
DF has been partially supported by the PRIN 2017 {\em Moduli and Lie
Theory}, Ref.\ 2017YRA3LK. He is a member of the {\em Grup\-po Nazionale per le
Strutture Algebriche, Geometriche e le loro Applicazioni} (GNSAGA-INdAM).

\section{The category $\mathbf{KR}$ of formal complexes of MOY graphs}\label{chapter:KR}

\subsection{Formal direct sums of shifted MOY-type graphs}
Morphisms in the category $\mathbf{KR}$ will be certain formal complexes of formal direct sums of shifted planar directed graphs of the kind considered by Murakami,  Ohtsuki and Yamada  in \cite{moy} (MOY type graphs for short). We will construct these in steps, starting with the graphs.
\begin{definition}
The collection of graphs of \emph{MOY type} is the collection of oriented planar diagrams with labelled edges  that can ``built'' out of the following elementary pieces:
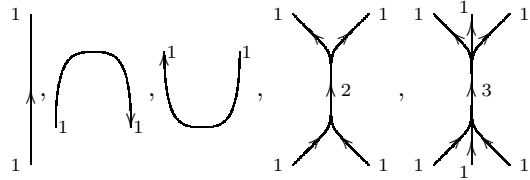
\[
\idone\,\,,\, \ev\,,\, \coev\,,\, \esse\,,\, \ti
\]
and their duals, i.e., the graphs obtained by reversing all the arrows in the graphs in the above list. The elementary pieces in the list above will be denoted by $\mathrm{id}_1,\mathrm{ev}_1,\mathrm{coev}_1, S$, and $T$, respectively.
\end{definition}

\begin{definition} A \emph{shifted} MOY-type graph is an object of the form $\Gamma\{m\}$ where $\Gamma$ is a MOY-type graph and $m$ is an integer, which we will call the ``shift'' in accordance with the terminology in \cite{Khovanov-Rozansky}. A formal sum of shifted MOY-type graphs is an expression of the form
\[
\oplus_i \Gamma_i\{m_i\}
\]
where $i$ ranges over a finite set of indices, and all the MOY-type graphs in the above expression have the same source and the same target data (i.e., the same source and target sequences of $\uparrow^1$'s and $\downarrow^1$'s).
\end{definition}

\begin{remark}
In \cite{Khovanov-Rozansky} the shift numbers correspond to actual degree shifts in complexes of matrix factorizations. Here they are just labels. By construction, the two shifts are identified by the map 
\[
\mathbf{KR}\xrightarrow{\mathrm{mf}}\mathcal{K}(\mathcal{LG})
\]
from  the combinatorial category $\mathbf{KR}$ to the ``derived Ginzburg-Landau category'' of homotopy classes of complexes of matrix factorizations $\mathcal{K}(\mathcal{LG})$ described in the Introduction.
\end{remark}
\begin{notation}
The label $0$ will usually be omitted, i.e., we will simply write $\Gamma$ for $\Gamma\{0\}$. 
\end{notation}
\begin{notation}
For every nonnegative integer $m$, we will write $[m]\,\Gamma$ for the formal direct sum
\[
[m]\,\Gamma=\Gamma\{1-m\}\oplus \Gamma\{3-m\}\oplus \cdots \oplus \Gamma\{m-1\}.
\]
When $\Gamma$ is the empty diagram we will often omit it from the notation, i.e., we will simply write $[m]$ for $[m]\emptyset$. Also, we will write $\mathbb{Q}\{m\}$ for $\emptyset\{m\}$, so that
\[
[m]=\mathbb{Q}\{1-m\}\oplus \mathbb{Q}\{3-m\}\oplus \cdots \oplus \mathbb{Q}\{m-1\}.
\]
\end{notation}
\begin{remark}
We have a natural notion of direct sum of formal sums of shifted MOY-type graphs, as well as a tensor product of formal sums of MOY-type graphs by finite dimensional graded $\mathbb{Q}$-vector spaces with graded bases (i.e., with a given isomorphism of graded  $\mathbb{Q}$-vector spaces to a graded  $\mathbb{Q}$-vector space of the form
\[
\oplus_i \mathbb{Q}^{n_i}\{m_i\},
\]
where $\{m_i\}$ denotes an actual degree shift here. The tensor product rule is, clearly,
\[
(\mathbb{Q}^k\{m\})\otimes \Gamma\{l\}=\underbrace{\Gamma\{m+l\}\oplus \Gamma\{m+l\}\oplus\cdots \Gamma\{m+l\}}_{k\text{ times}},
\]
extended by bilinearity to formal direct sums of MOY-type graphs and direct sums of graded $\mathbb{Q}$-vector spaces.

\end{remark}

\begin{remark}
We can extend the composition and tensor product of MOY-type graphs (defined by concatenation and juxtaposition) to shifthed MOY-type graphs by the rule
\[
\Gamma\{m\}\circ \Phi\{k\}=\Gamma\circ\Phi\{m+k\}; \qquad \Gamma\{m\}\otimes \Phi\{k\}=\Gamma\otimes\Phi\{m+k\}.
\]
As all MOY-type graphs in a formal direct sum have the same source sequence and the same target sequence of of $\uparrow^1$'s and $\downarrow^1$'s, we can extend these operatins by bilinearity and define a composition $M\circ N$ and a tensor product $M\otimes N$ of formal direct sums $M$ and $N$ of shifted MOY-type graphs. All the graphs appearing in  $M\circ N$ and in $M\otimes N$ will be shifted copies of MOY-type graphs all having the same source sequence and the same target sequence of of $\uparrow^1$'s and $\downarrow^1$'s; so both $M\circ N$ and  $M\otimes N$ will be again formal direct sums of shifted MOY-type graphs.
\end{remark}

\subsection{Morphisms between formal direct sums}
The next step consists in defining a set of morphisms between the formal direct sums of MOY-type graphs introduced in the previous section. Since we may think of formal direct sums of MOY-type graphs as morphisms between sequences of $\uparrow^1$'s and $\downarrow^1$'s, this step suggests we are here going towards the construction of a 2-category. This is indeed so. However, we will not be interested in constructing such a 2-category in full detail, as we are going to eventually decategorify it by taking suitable equivalence classes of 2-morphisms. Therefore, we will not need completely constructing the 2-category: we will directly introduce certain equivalence classes of complexes of formal sums of shifted MOY-type graphs which will be our 1-morphisms between sequences of $\uparrow^1$'s and $\downarrow^1$'s. These can be thought of as equivalence classes of 2-morphisms, although we actually do not really construct the 2-category. Before introducing complexes of (formal sums of shifted) MOY-type graphs, we need define what the arrows in these complexes (i.e., the morphisms between MOY-type graphs) would be. This is done in several steps.
\subsection{Elementary morphisms}
In this section we introduce a set of generators for the morphisms in a category of MOY-type graphs. We will call these generators ``elementary morphism'', although at this stage they are not, of course, yet morphisms in a category. The idea is that the morphisms in the category we are going to describe will be ``built up'' starting from these ``elementary morphisms''. Sources and targets of these elementary morphisms will be formal directed sums of shifted  MOY-type graphs as defined in the previous section.
Finally $n$ will be a fixed integer with $n\geq 2$.

\begin{definition}\label{def:generating}
The set $\mathcal{E}$ of \emph{elementary morphisms} is the collection of arrows $\{1_\Gamma, i^u_\Gamma,i_{d;\Gamma}, \chi_0,\chi_1,\alpha,\gamma,\lambda, \epsilon, \mu, \varphi,\eta \}$, with $\Gamma$ ranging among all MOY graphs,
where
\begin{itemize}
\item $1_\Gamma$ is the identity arrow
\[
1_\Gamma\colon \Gamma \to \Gamma
\]
(we will often denote this arrow just by $1$ leaving $\Gamma$ be understood by the context)
\item $i^u_\Gamma$ and $i_{d;\Gamma}$ are two invertible arrows
\[
i^u_\Gamma\colon \Gamma \xrightarrow{\sim}  \mathrm{id}_{t(\Gamma)}\circ \Gamma; \qquad \qquad i_{d;\Gamma}\colon \Gamma \xrightarrow{\sim} \Gamma\circ \mathrm{id}_{s(\Gamma)}
\]
where $s(\Gamma)$ and $t(\Gamma)$ denote the source and the target of $\Gamma$, respectively.
(we will often denote these arrows just by $i^u$ and $i_d$, leaving $\Gamma$ be understood by the context)
\item $\chi_0$ is a distinguishes arrow
\[
\chi_0\colon \mathrm{id}_2\to S\{1\},
\]
where $\mathrm{id}_2=\mathrm{id}_1\otimes \mathrm{id}_1$,
i.e., in graphical notation,
\[
\idtwo\xrightarrow{\phantom{mm}\chi_0\phantom{mm}} \,
\esse\{1\};
\]
\item $\chi_1$ is a distinguishes arrow
\[
\chi_1\colon S\{-1\}\to \mathrm{id}_2,
\]
i.e., in graphical notation,
\[
\esse\{-1\}
\xrightarrow{\phantom{mm}\chi_1\phantom{mm}} \,
\idtwo;
\]

\item $\epsilon$ is a distinguished arrow
\[
\mathrm{id}_1\{-1\}\xrightarrow{\epsilon} \mathrm{id}_1\{1\}
\]
i.e., in graphical notation,
\[
\epsilon\colon \idone\,\,\{-1\}\to \idone\,\,\{1\};
\]

\item $\lambda$ is a distinguished invertible arrow
\[
\mathrm{ev}_1\circ \mathrm{coev}_1\xrightarrow[\sim]{\lambda} [n]=[n]\,\emptyset,
\]
i.e., in graphical notation,
\[
\lambda\colon \iunknot\xrightarrow{\sim} [n]=[n]\,\emptyset
\]

\item $\mu$ is a distinguished invertible arrow
\[
\mu\colon (\mathrm{ev}_1\otimes \mathrm{id}_1)\circ(\mathrm{id}_1\otimes S)\circ (\mathrm{coev}_1\otimes \mathrm{id}_1)\xrightarrow{\sim}
[n-1]\, \mathrm{id}_1,
\]
i.e., in graphical notation
\[
{\xy
(0,-10);(0,10)**\crv{(0,-7)&(0,-6)&(0,-5)&(0,0)&(0,5)&(0,6)&(0,7)}
?>(.2)*\dir{>}?>(.99)*\dir{>}?>(.55)*\dir{>}
,(0,-3);(0,3)**\crv{(0,-4)&(0,-5)&(-4,-6)&(-6,0)&(-4,6)&(0,5)&(0,4)}
?>(.55)*\dir{<}
,(2,-11)*{\scriptstyle{1}}
,(2,11)*{\scriptstyle{1}}
,(2,0)*{\scriptstyle{2}}
,(-6,0)*{\scriptstyle{1}}
\endxy}\,\,\,\xrightarrow[\sim]{\phantom{mm}\mu\phantom{mm}}
[ n-1]\,\,{\xy
(0,-10);(0,10)**\dir{-}
?>(.95)*\dir{>}
,(2,-11)*{\scriptstyle{1}}
,(2,11)*{\scriptstyle{1}}
\endxy}
\]
\item $\varphi$ and $\psi$ are two a distinguished invertible arrows
\[
\varphi,\psi\colon S\circ S\xrightarrow{\sim} S\{-1\}\oplus S\{1\},
\]
i.e., in graphical notation
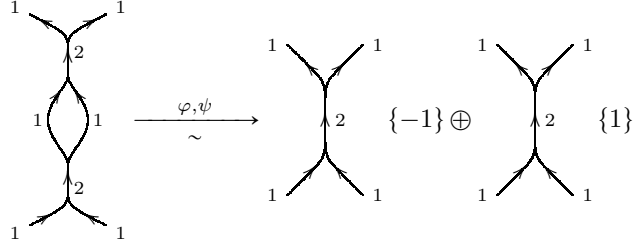
\[
\essecircesse\xrightarrow[\sim]{\phantom{mm}\varphi,\psi\phantom{mm}}
\esse\{-1\}
\oplus
\esse\{1\}
\]

\item $\alpha,\gamma$  are distinguished arrows
\[
\alpha,\gamma\colon S\{-1\}\to S\{1\},
\]
i.e., in graphical notation
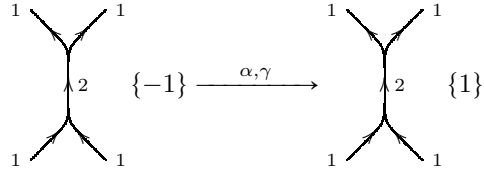
\[
\esse\{-1\}\xrightarrow{\phantom{mm}\alpha,\gamma\phantom{mm}}
\esse\{1\}
\]

\item $\nu$ is a distinguished invertible arrow
\[
\nu\colon (\mathrm{id}_1\otimes \mathrm{ev}_1\otimes\mathrm{id}_1^\vee)\circ (S\otimes S^\vee) \circ (\mathrm{id}_1\otimes \mathrm{coev}_1\otimes\mathrm{id}_1^\vee)
\xrightarrow{\sim} \mathrm{coev}_1\circ\mathrm{ev}_1\oplus [n-2]\otimes \mathrm{id}_1\otimes \mathrm{id}_1^\vee
\]
i.e., in graphical notation
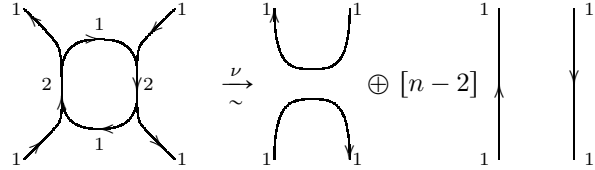
\[
{\xy
(-10,-10);(-10,10)**\crv{(-6,-6)&(-5,-5)&(-5,-2)&(-5,-2)&(-5,5)&(-6,6)}
?>(.05)*\dir{>}?>(.99)*\dir{>}?>(.55)*\dir{>}
,(10,-10);(10,10)**\crv{(6,-6)&(5,-5)&(5,-2)&(5,2)&(5,5)&(6,6)}
?>(.03)*\dir{<}?>(.92)*\dir{<}?>(.45)*\dir{<}
,(-5,0);(5,0)**\crv{(-5,-2)&(-5,-3)&(-4,-5)&(-2,-6)&(2,-6)&(4,-5)&(5,-3)&(5,-2)}?>(.5)*\dir{<}
,(-5,0);(5,0)**\crv{(-5,2)&(-5,3)&(-4,5)&(-2,6)&(2,6)&(4,5)&(5,3)&(5,2)}?>(.5)*\dir{>}
,(-11,-10)*{\scriptstyle{1}}
,(-11,10)*{\scriptstyle{1}}
,(11,-10)*{\scriptstyle{1}}
,(11,10)*{\scriptstyle{1}}
,(-7,0)*{\scriptstyle{2}}
,(0,-8)*{\scriptstyle{1}}
,(6.5,0)*{\scriptstyle{2}}
,(0,8)*{\scriptstyle{1}}
\endxy}
\quad
\xrightarrow[\sim]{\nu}
\xy
(-5,-10);(5,-10)**\crv{(-5,-2)&(0,-2)&(5,-2)}
?>(1)*\dir{>},
(-6,-10)*{\scriptstyle{1}},(6,-10)*{\scriptstyle{1}}
,(-5,10);(5,10)**\crv{(-5,2)&(0,2)&(5,2)}
?>(0)*\dir{<},
(-6,10)*{\scriptstyle{1}},(6,10)*{\scriptstyle{1}}
\endxy
\oplus\, [n-2]
{\xy
(-5,-10);(-5,10)**\dir{-}?>(.5)*\dir{>}
,(5,-10);(5,10)**\dir{-}?>(.5)*\dir{<}
,(-7,-10)*{\scriptstyle{1}}
,(7,-10)*{\scriptstyle{1}}
,(-7,10)*{\scriptstyle{1}}
,(7,10)*{\scriptstyle{1}}
\endxy}
\]
\item $\eta$ is a distinguished invertible arrow
\[
\eta\colon (S\otimes\mathrm{id}_1)\circ (\mathrm{id}_1\otimes S) \circ (S\otimes\mathrm{id}_1)\xrightarrow{\sim} (S\otimes \mathrm{id}_1)\oplus T
\]
i.e., in graphical notation
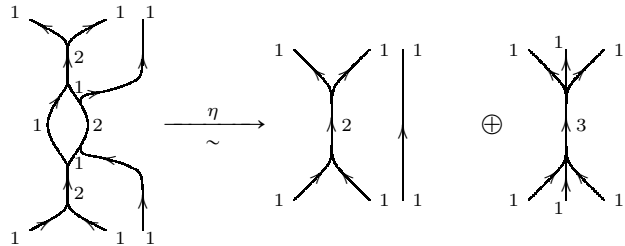
\[
\coso
\xrightarrow[\sim]{\phantom{mm}\eta\phantom{mm}}
\esse\,\,\idoner\qquad
\oplus \ti
\]
\end{itemize}
The set $\mathcal{E}$ also contains all the vertical and horizontal reflections of the above arrows, which we will denote by the same symbols. For instance, we also have in $\mathcal{E}$ the distinguished invertible arrow
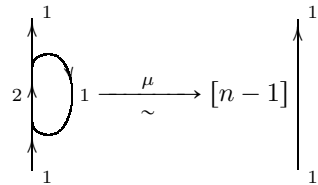
\[
{\xy
(0,-10);(0,10)**\crv{(0,-7)&(0,-6)&(0,-5)&(0,0)&(0,5)&(0,6)&(0,7)}
?>(.2)*\dir{>}?>(.99)*\dir{>}?>(.55)*\dir{>}
,(0,-3);(0,3)**\crv{(0,-4)&(0,-5)&(4,-6)&(6,0)&(4,6)&(0,5)&(0,4)}
?>(.55)*\dir{<}
,(2,-11)*{\scriptstyle{1}}
,(2,11)*{\scriptstyle{1}}
,(-2,0)*{\scriptstyle{2}}
,(7,0)*{\scriptstyle{1}}
\endxy}\xrightarrow[\sim]{\phantom{mm}\mu\phantom{mm}}
[ n-1]\,\,{\xy
(0,-10);(0,10)**\dir{-}
?>(.95)*\dir{>}
,(2,-11)*{\scriptstyle{1}}
,(2,11)*{\scriptstyle{1}}
\endxy}
\]
etc. 
\end{definition}

We will usually omit the morphisms $i^u$ and $i_d$ from the notation. So, for instance, we will simply write
\[
S\xrightarrow{\chi_0\circ 1} S\circ S\{1\}
\]
to mean the composition
\[
S\xrightarrow{i^u_S}\mathrm{id}_2\circ S\xrightarrow{\chi_0\circ 1}S\circ S\{1\}
\]
Graphically, this means that we simply write
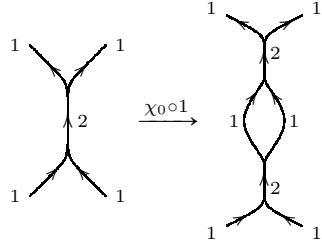
\[
\esse
\xrightarrow{\chi_0\circ 1} 
\essecircesse
\]
to actually mean the composition
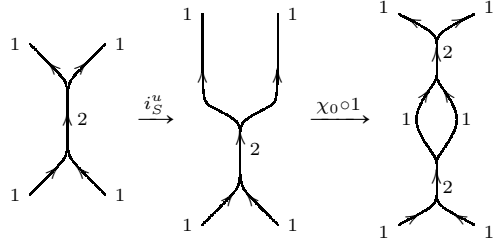
\[
\esse
\xrightarrow{i^u_S}
{\xy
(-5,-14);(-5,14)**\crv{(-1,-10)&(0,-9)&(0,-6)&(0,-2)&(0,0)&(-5,2)&(-5,3)}
?>(.55)*\dir{>}?>(.1)*\dir{>}?>(.95)*\dir{>}
,(5,-14);(5,14)**\crv{(1,-10)&(0,-9)&(0,-6)&(0,-2)&(0,0)&(5,2)&(5,3)}
?>(.1)*\dir{>}?>(.95)*\dir{>}
,(-7,-14)*{\scriptstyle{1}}
,(7,-14)*{\scriptstyle{1}}
,(-7,14)*{\scriptstyle{1}}
,(7,14)*{\scriptstyle{1}}
,(2,-4)*{\scriptstyle{2}}
\endxy}\,
\xrightarrow{\chi_0\circ 1} 
\essecircesse
\]
Also, when there are no ambiguities, we will also omit the identity morphisms, so for instance we will simply write
\[
\iunknot\,\,\{-1\}\xrightarrow{\epsilon} \iunknot\,\,\{1\}
\]
to mean the composition
\[
\iunknot\,\,\{-1\} \xrightarrow[\sim]{1\circ i^u_{\mathrm{coev}}} {\xy
(0,-10);(0,-10)**\crv{(5,-10)&(5,10)&(-5,10)&(-5,-10)}
?>(.3)*\dir{>},(-4,-9)*{\scriptstyle{1}}
\endxy}\,\,\{-1\} \xrightarrow{1\circ(1\otimes \epsilon)\circ 1} {\xy
(0,-10);(0,-10)**\crv{(5,-10)&(5,10)&(-5,10)&(-5,-10)}
?>(.3)*\dir{>},(-4,-9)*{\scriptstyle{1}}\endxy}\,\,\{1\}\xrightarrow[\sim]{1\circ (i^u_{\mathrm{coev}})^{-1}}\iunknot\,\,\{1\}
\]
Finally, if $f\colon \Gamma_1\to \Gamma_2$ is an elementary morphism, we denote by the same symbol its ``shifted'' versions, i.e., we simply write $f\colon \Gamma_1\{m\}\to \Gamma_2\{m\}$ instead of the more precise $f\{m\}\colon \Gamma_1\{m\}\to \Gamma_2\{m\}$.

\begin{remark}
Notice that if $f\colon \Gamma_1\to \Gamma_2$ is a morphism between (formal direct sums of shifted) MOY-type graphs, then $\Gamma_1$ and $\Gamma_2$ have the same incoming and the same outgoing sequences of $\uparrow^1$'s and $\downarrow^1$'s. This is coherent with the idea we are (partially) constructing a 2-categorification of the category of graphs considered by Murakami,  Ohtsuki and Yamada  in \cite{moy}. 
\end{remark}

\subsection{Composite morphisms}
Starting from the elementary morphisms we can consider the ${\mathbb{Q}}$-vector space generated by the elements in  $\mathcal{E}$ by $\mathbb{Q}$-linear combinations, compositions and direct sums. More precisely, we give the following
\begin{definition}
The set $\mathcal{M}$ of morphisms between formal direct sums of MOY-type graphs is the smallest set with the following properties:
\begin{itemize}
\item $\mathcal{E}\subseteq \mathcal{M}$;

\item if $f,g\colon \Gamma_1\to \Gamma_2$ are $\mathcal{M}$, then for any $a,b\in \mathbb{Q}$, the element $a\,g+b\,g\colon \Gamma_1\to \Gamma_2$ is in $\mathcal{M}$ ($\mathcal{M}$ is a $\mathbb{Q}$-vector space);

\item if $f\colon \Gamma_1\to \Gamma_2$ and $g\colon \Gamma_2\to \Gamma_3$ are in $\mathcal{M}$, then $gf\colon \Gamma_1\to \Gamma_3$ is in $\mathcal{M}$ (horizontal composition);

\item if $f\colon \Gamma_1\to \Gamma_2$ and $g\colon \Psi_1\to \Psi_2$ are in $\mathcal{M}$, and the compositions $\Psi_i\circ\Gamma_i$ are defined,  then $g\circ f\colon \Psi_1\circ\Gamma_1\to \Psi_2\circ\Gamma_2$ is in $\mathcal{M}$ (vertical composition);

\item if $f\colon \Gamma_1\to \Gamma_2$ is in $\mathcal{M}$ and $m\in \mathbb{Z}$, then $f\{m\}\colon \Gamma_1\{m\}\to \Gamma_2\{m\}$ is in $\mathcal{M}$;

\item  if $f\colon \Gamma_1\to \Gamma_2$ and $g\colon \Psi_1\to \Psi_2$ are in $\mathcal{M}$, then $f\oplus g\colon \Gamma_1\oplus \Psi_1\to \Gamma_2\oplus \Psi_2$ is in $\mathcal{M}$.
\end{itemize}
In the above conditions, $\Gamma_i$ and $\Phi_i$ denote arbitary formal direct sums of shifted MOY-type graphs.
\end{definition}

\begin{remark}
By the above definition, the set $\mathcal{M}$ can be  recursively built from $\mathcal{E}$ by considering morphisms of increasing complexity. 
\end{remark}

\subsection{Relations between the elementary morphisms}\label{sec:relations}
As we want to think of $\mathcal{M}$ as the set of morphisms between formal direct sums of MOY-type graphs we need to impose associativity relations and identity element axiom on the horizontal composition in $\mathcal{M}$, i.e., we impose the relations
\[
(fg)h=f(gh); \qquad f\,1_\Gamma=1_\Phi\, f=f
\]
for any composable $f,g,h,1_\Gamma,1\Phi$ in $\mathcal{M}$. Also, as we want to think of the above construction as part of the construction of a 2-categorification of the category of graphs considered by Murakami,  Ohtsuki and Yamada, we also impose associativity and identity element axiom for the vertical composition as well as the compatibility between vertical and horizontal composition of elements in $\mathcal{M}$.
Yet, even after having quotiented out the associative, identity and compatibility relations, $\mathcal{M}$ is too big for any meaningful use: it is just the set of morphisms freely generated by our collection of elementary morphisms. To turn it in something that can actually be used to define a rigid braided monoidal category, and so ultimately to define link invariants we impose additional relations on the elements in $\mathcal{M}$.
\begin{remark}
In view of the higher categorical inspiration of the construction, these relations can be thought of as 3-morphism, so in a sense we are going to construct the 1-category $\mathbf{KR}$ by providing sketches of a 3-category that in the end we decategorify twice.
\end{remark}  
As $\mathcal{M}$ is generated by $\mathcal{E}$, in order to impose additional relations on $\mathcal{M}$ we impose a set of additional relations on certain compositions of elementary morphisms. The whole set of the additional relations will the be the set generated by these elementary relations.
\vskip .7 cm
The set of additional relations between compositions of elementary morphisms that we impose abstractly encodes all of the properties of complexes of matrix factorizations used by Khovanov and Rozansky in their proof of the invariance of level $n$ Khovanov homology by Reidemeister moves \cite{Khovanov-Rozansky}, as well as additional properties demonstrated by Rasmussen in \cite{Rasmussen}. It is the following set of relations:
\begin{itemize}
\item If $\mathbf{0}$ denotes the empty direct sum, then the diagrams
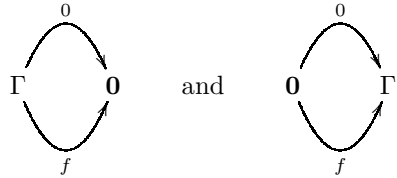
\[
\xymatrix{
\Gamma \ar@/_2pc/[r]_{f} \ar@/^2pc/[r]^{0}&\mathbf{0}
}
\qquad \text{and}\qquad
\xymatrix{
\mathbf{0} \ar@/_2pc/[r]_{f} \ar@/^2pc/[r]^{0}&\Gamma
}
\]
commute for any $f$ (i.e., the only morphism to and from $\mathbf{0}$ is the zero morphism);

\item For any wo composable MOY garphs, $\Phi$ and $\Gamma$  the following diagram commutes:
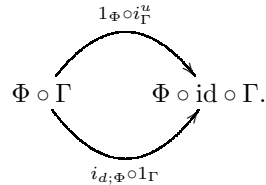
\[
\xymatrix{
\Phi\circ\Gamma \ar@/_2pc/[r]_{i_{d;\Phi}\circ 1_\Gamma} \ar@/^2pc/[r]^{1_\Phi\circ i^u_\Gamma}&\Phi\circ\mathrm{id}\circ\Gamma.
}
\] 

\item The following diagrams commute:
\begin{enumerate}
\item 
\[
\xymatrix{
\idtwo\{-1\}\ar[r]^-{\chi_0} \ar@/_4pc/[rr]_{0}& \esse\ar[r]^-{\chi_1}& \idtwo\{1\},
}
\]
i.e., we have $\chi_1\chi_0=0$;

\item 
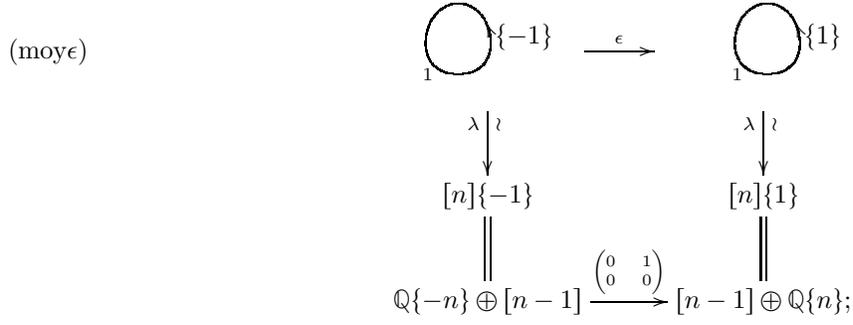
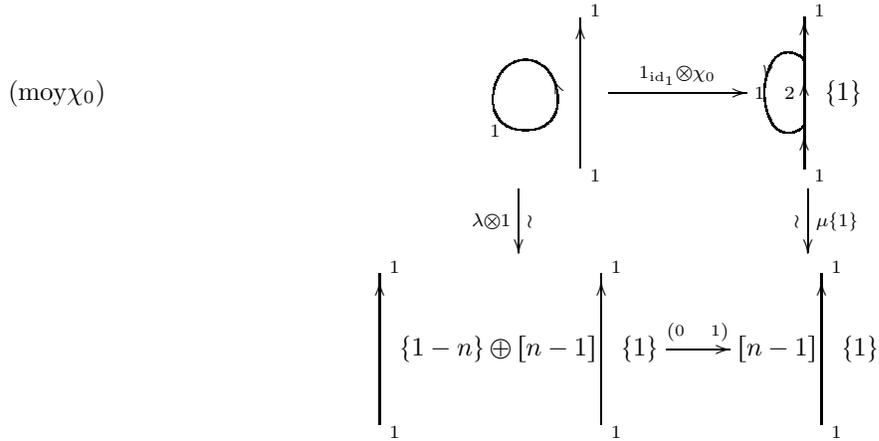
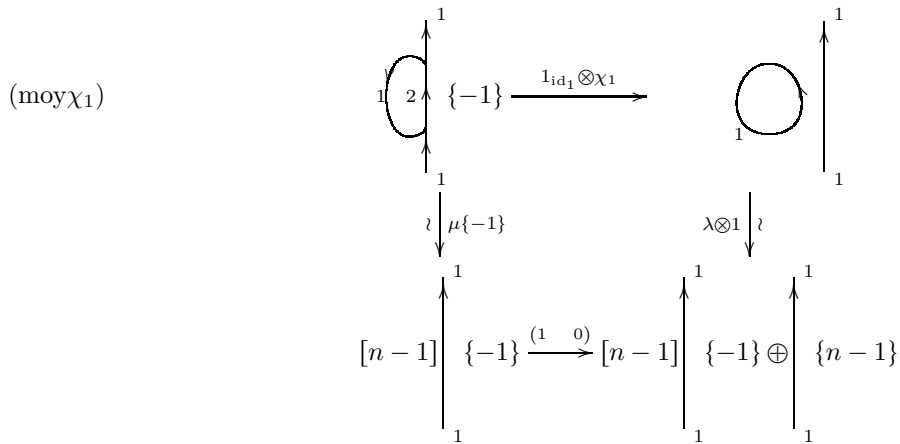
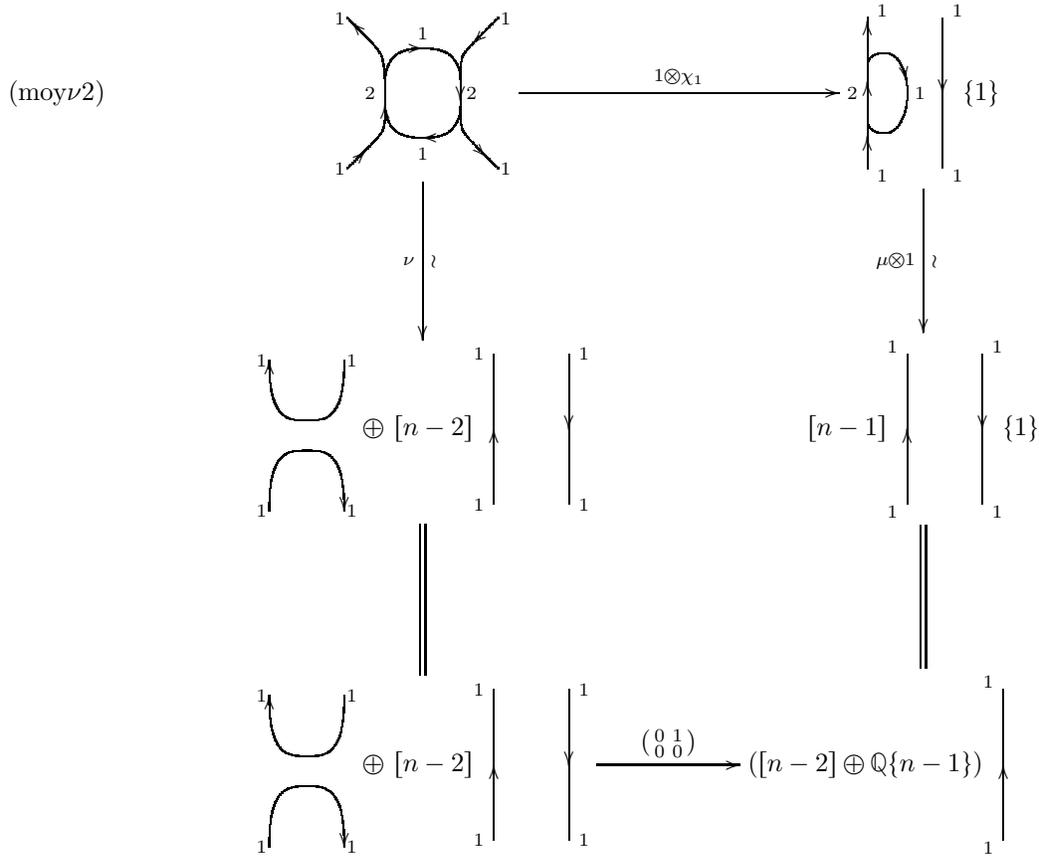
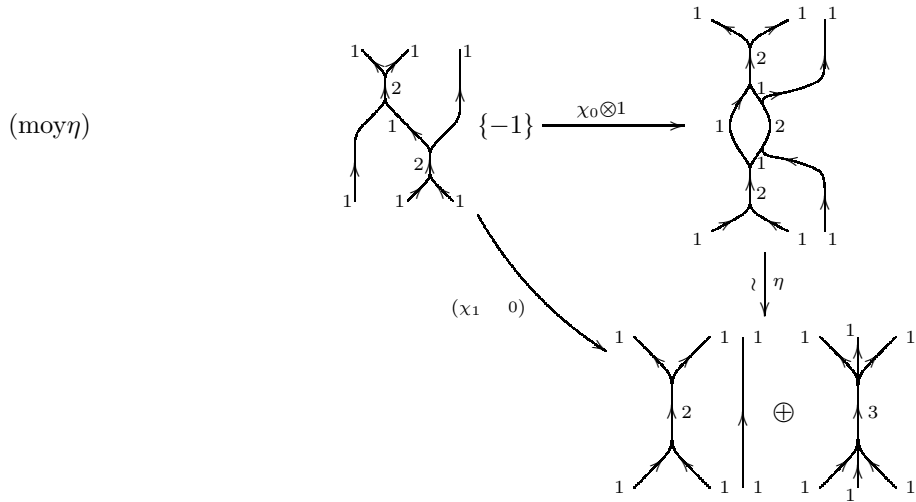
\begin{equation}\tag{moy$\epsilon$}\label{eq:moyepsilon}
\xymatrix{
*++++{\iunknot\,\, \{-1\}} \ar[r]^-{\epsilon}\ar[d]_{\lambda}^{\wr} &
*++++{\phantom{mm}\iunknot\,\, \{1\}}\ar[d]_{\lambda}^{\wr}\\
[n]\{-1\}\ar@{=}[d] & [n]\{1\}\ar@{=}[d]\\
\mathbb{Q}\{-n\}\oplus [n-1] 
\ar[r]^-{\tiny{\left(\begin{matrix}0&1\\0 & 0\end{matrix}\right)}}&
[n-1]\oplus\mathbb{Q}\{n\} ;
}
\end{equation}

\item 
\begin{equation}\tag{moy$\chi_0$}\label{eq:moychi0}
\xymatrix{
\qquad\iunknot \quad {\xy
(0,-10);(0,10)**\dir{-}
?>(.95)*\dir{>}
,(2,-11)*{\scriptstyle{1}}
,(2,11)*{\scriptstyle{1}}
\endxy} \ar[r]^-{1_{\mathrm{id}_1}\otimes \chi_0}\ar[d]_-{\lambda\otimes 1}^-{\wr} &
{\xy
(0,-10);(0,10)**\crv{(0,-7)&(0,-6)&(0,-5)&(0,0)&(0,5)&(0,6)&(0,7)}
?>(.2)*\dir{>}?>(.99)*\dir{>}?>(.55)*\dir{>}
,(0,-3);(0,3)**\crv{(0,-4)&(0,-5)&(-4,-6)&(-6,0)&(-4,6)&(0,5)&(0,4)}
?>(.55)*\dir{<}
,(2,-11)*{\scriptstyle{1}}
,(2,11)*{\scriptstyle{1}}
,(-2,0)*{\scriptstyle{2}}
,(-6,0)*{\scriptstyle{1}}
\endxy}\{1\}\ar[d]^{\mu\{1\}}_{\wr}\\
{\xy
(0,-10);(0,10)**\dir{-}
?>(.95)*\dir{>}
,(2,-11)*{\scriptstyle{1}}
,(2,11)*{\scriptstyle{1}}
\endxy} \{1-n\}\oplus [n-1]\,{\xy
(0,-10);(0,10)**\dir{-}
?>(.95)*\dir{>}
,(2,-11)*{\scriptstyle{1}}
,(2,11)*{\scriptstyle{1}}
\endxy} \{1\}
\ar[r]^-{\tiny{\left(\begin{matrix}0&1\end{matrix}\right)}}& [n-1]\,{\xy
(0,-10);(0,10)**\dir{-}
?>(.95)*\dir{>}
,(2,-11)*{\scriptstyle{1}}
,(2,11)*{\scriptstyle{1}}
\endxy} \{1\}
}
\end{equation}

\item 
\begin{equation}\tag{moy$\chi_1$}\label{eq:moychi1}
\xymatrix{
{\xy
(0,-10);(0,10)**\crv{(0,-7)&(0,-6)&(0,-5)&(0,0)&(0,5)&(0,6)&(0,7)}
?>(.2)*\dir{>}?>(.99)*\dir{>}?>(.55)*\dir{>}
,(0,-3);(0,3)**\crv{(0,-4)&(0,-5)&(-4,-6)&(-6,0)&(-4,6)&(0,5)&(0,4)}
?>(.55)*\dir{<}
,(2,-11)*{\scriptstyle{1}}
,(2,11)*{\scriptstyle{1}}
,(-2,0)*{\scriptstyle{2}}
,(-6,0)*{\scriptstyle{1}}
\endxy}\{-1\}\ar[d]^{\mu\{-1\}}_{\wr}
 \ar[r]^-{1_{\mathrm{id}_1}\otimes \chi_1}
&
\qquad\quad \iunknot \quad {\xy
(0,-10);(0,10)**\dir{-}
?>(.95)*\dir{>}
,(2,-11)*{\scriptstyle{1}}
,(2,11)*{\scriptstyle{1}}
\endxy}\ar[d]_-{\lambda\otimes 1}^-{\wr} 
\\
[n-1]\,{\xy
(0,-10);(0,10)**\dir{-}
?>(.95)*\dir{>}
,(2,-11)*{\scriptstyle{1}}
,(2,11)*{\scriptstyle{1}}
\endxy} \{-1\}\ar[r]^-{\tiny{\left(\begin{matrix}1&0\end{matrix}\right)}}&
[n-1]\,{\xy
(0,-10);(0,10)**\dir{-}
?>(.95)*\dir{>}
,(2,-11)*{\scriptstyle{1}}
,(2,11)*{\scriptstyle{1}}
\endxy}\{-1\}\oplus {\xy
(0,-10);(0,10)**\dir{-}
?>(.95)*\dir{>}
,(2,-11)*{\scriptstyle{1}}
,(2,11)*{\scriptstyle{1}}
\endxy} \{n-1\} 
 }
\end{equation}

\item
\[
\xymatrix{
\esseb\{-1\}\ar[r]^-{\chi_0\circ 1} \ar@/_2pc/[rr]_{\alpha}& \essecircesseb\ar[r]^-{1\circ \chi_1}& \esseb\{1\}
}
\]
In other words, $\alpha=\chi_0\circ \chi_1$, i.e., $\alpha$ is the composition of $\chi_0$ and $\chi_1$. One could object that $\alpha$ is not elementary, as it is  a suitable compositions of $\chi_0$ and $\chi_1$. However, as it will show up several times in what follows, we find it convenient to think of $\alpha$ as an elementary morphism, which additionally enjoys a nice relationship with the other two elementary morphisms $\chi_0$ and $\chi_1$.

\item 
\begin{equation}\tag{moy$\chi_0$a}\label{eq:moichi0a}
\xymatrix{
\qquad{\xy
(0,-10);(0,10)**\crv{(0,-7)&(0,-6)&(0,-5)&(0,0)&(0,5)&(0,6)&(0,7)}
?>(.2)*\dir{>}?>(.99)*\dir{>}?>(.55)*\dir{>}
,(0,-3);(0,3)**\crv{(0,-4)&(0,-5)&(-4,-6)&(-6,0)&(-4,6)&(0,5)&(0,4)}
?>(.55)*\dir{<}
,(2,-11)*{\scriptstyle{1}}
,(2,11)*{\scriptstyle{1}}
,(-2,0)*{\scriptstyle{2}}
,(-6,0)*{\scriptstyle{1}}
\endxy}\{-1\} \ar[r]^-{1_{\mathrm{id}_1}\otimes \alpha}\ar[d]_-{\mu\{-1\}}^-{\wr} &
{\xy
(0,-10);(0,10)**\crv{(0,-7)&(0,-6)&(0,-5)&(0,0)&(0,5)&(0,6)&(0,7)}
?>(.2)*\dir{>}?>(.99)*\dir{>}?>(.55)*\dir{>}
,(0,-3);(0,3)**\crv{(0,-4)&(0,-5)&(-4,-6)&(-6,0)&(-4,6)&(0,5)&(0,4)}
?>(.55)*\dir{<}
,(2,-11)*{\scriptstyle{1}}
,(2,11)*{\scriptstyle{1}}
,(-2,0)*{\scriptstyle{2}}
,(-6,0)*{\scriptstyle{1}}
\endxy}\{1\}\ar[d]^{\mu\{1\}}_{\wr}\\
 [n-1]\,{\xy
(0,-10);(0,10)**\dir{-}
?>(.95)*\dir{>}
,(2,-11)*{\scriptstyle{1}}
,(2,11)*{\scriptstyle{1}}
\endxy} \{-1\}
\ar[r]^-{0}& [n-1]\,{\xy
(0,-10);(0,10)**\dir{-}
?>(.95)*\dir{>}
,(2,-11)*{\scriptstyle{1}}
,(2,11)*{\scriptstyle{1}}
\endxy} \{1\}
}
\end{equation}

\item 
\begin{equation}\tag{moy$\chi_0$b}\label{eq:moichi0b}
\xymatrix{
\qquad{\xy
(0,-10);(0,10)**\crv{(0,-7)&(0,-6)&(0,-5)&(0,0)&(0,5)&(0,6)&(0,7)}
?>(.2)*\dir{>}?>(.99)*\dir{>}?>(.55)*\dir{>}
,(0,-3);(0,3)**\crv{(0,-4)&(0,-5)&(-4,-6)&(-6,0)&(-4,6)&(0,5)&(0,4)}
?>(.55)*\dir{<}
,(2,-11)*{\scriptstyle{1}}
,(2,11)*{\scriptstyle{1}}
,(-2,0)*{\scriptstyle{2}}
,(-6,0)*{\scriptstyle{1}}
\endxy}\{-1\} \ar[r]^-{1_{\mathrm{id}_1}\otimes \gamma}\ar[d]_-{\mu\{-1\}}^-{\wr} &
{\xy
(0,-10);(0,10)**\crv{(0,-7)&(0,-6)&(0,-5)&(0,0)&(0,5)&(0,6)&(0,7)}
?>(.2)*\dir{>}?>(.99)*\dir{>}?>(.55)*\dir{>}
,(0,-3);(0,3)**\crv{(0,-4)&(0,-5)&(-4,-6)&(-6,0)&(-4,6)&(0,5)&(0,4)}
?>(.55)*\dir{<}
,(2,-11)*{\scriptstyle{1}}
,(2,11)*{\scriptstyle{1}}
,(-2,0)*{\scriptstyle{2}}
,(-6,0)*{\scriptstyle{1}}
\endxy}\{1\}\ar[d]^{\mu\{1\}}_{\wr}\\
 [n-1]\,{\xy
(0,-10);(0,10)**\dir{-}
?>(.95)*\dir{>}
,(2,-11)*{\scriptstyle{1}}
,(2,11)*{\scriptstyle{1}}
\endxy} \{-1\}
\ar[r]^-{\epsilon}& [n-1]\,{\xy
(0,-10);(0,10)**\dir{-}
?>(.95)*\dir{>}
,(2,-11)*{\scriptstyle{1}}
,(2,11)*{\scriptstyle{1}}
\endxy} \{1\}
}
\end{equation}

\item 
\begin{equation}\tag{moy2a}\label{eq:moy2a}
\xymatrix{
\esseb\{-1\}
\ar@/_1pc/[dr]_-{\left(\begin{smallmatrix}
1\\ 0
\end{smallmatrix}\right)} \ar[r]^{1\circ \chi_0}&
\quad
\essecircesseb
\ar^-{\varphi}[d]\\
 &
\esseb\{-1\} \oplus \esseb\{1\}
}\end{equation}

\item 
\begin{equation}\tag{moy2b}\label{eq:moy2b}
\xymatrix{
\esseb\{-1\}
\ar@/_1pc/[dr]_-{\left(\begin{smallmatrix}
1\\ \alpha
\end{smallmatrix}\right)} \ar[r]^{\chi_0\circ 1}&
\quad
\essecircesseb
\ar^-{\varphi}[d]\\
 &
\esseb\{-1\} \oplus \esseb\{1\}
}\end{equation}

\item 
\begin{equation}\tag{moyz}\label{eq:moyz}
\xymatrix{
\esse\,\,\idoner
 \ar@/_4pc/[rr]_{0} \ar@/^4pc/[rr]^{(1_S\otimes 1)\circ(1\otimes\chi_0 )}&&\zorro\{1\}
 },
\end{equation}
i.e., we have
\[
(1_S\otimes 1)\circ(1\otimes\chi_0 )=0;
\]

\item 
\begin{equation}\tag{moy2c}\label{eq:moy2c}
\xymatrix{
\esseb\{-1\}
\ar@/_1pc/[dr]_-{\left(\begin{smallmatrix}
1\\ \gamma
\end{smallmatrix}\right)} \ar[r]^{\chi_0\circ 1}&
\quad
\essecircesseb
\ar^-{\psi}[d]\\
 &
\esseb\{-1\} \oplus \esseb\{1\}
}\end{equation}

\item 
\begin{equation}\tag{shift1}\label{eq:shift}
\xymatrix{
\essecircesseb\ar[r]^-{\varphi}\ar[d]_{1\circ\alpha}& \esseb\{-1\} \oplus \esseb\{1\}\ar[d]^{\tiny{\left(\begin{matrix}0 & 1\\ 0 & 0\end{matrix}\right)}}\\
\essecircesseb\{2\}\ar[r]^-{\psi}&\esseb\{1\} \oplus \esseb\{3\}
}
\end{equation}

\item 
\begin{equation}\tag{shift2}\label{eq:shift2}
\xymatrix{
\essecircesseb\ar[r]^-{\psi}\ar[d]_{1\circ\gamma}& \esseb\{-1\} \oplus \esseb\{1\}\ar[d]^{\tiny{\left(\begin{matrix}0 & 1\\ 0 & 0\end{matrix}\right)}}\\
\essecircesseb\{2\}\ar[r]^-{\varphi}&\esseb\{1\} \oplus \esseb\{3\}
}
\end{equation}

\item 
\begin{equation}\tag{moy2d}\label{eq:moy2d}
\xymatrix{
\essecircesseb\ar[d]_{\varphi}\ar[r]^{1\circ\chi_1}&\esseb\{1\}\\
\esseb\{-1\} \oplus \esseb\{1\}\ar@/_1pc/[ru]_-{\tiny{\left(\begin{matrix}0&1\end{matrix}\right)}}
}
\end{equation}

\item 
\begin{equation}
\label{eq:moynu1}\tag{moy$\nu$1}
\xymatrix{
{\xy
(-10,-10);(-10,10)**\dir{-}?>(.5)*\dir{>}
,(-12,-11)*{\scriptstyle{1}}
,(-12,11)*{\scriptstyle{1}}
,(0,10);(0,-10)**\crv{(0,7)&(0,6)&(0,5)&(0,0)&(0,-5)&(0,-6)&(0,-7)}
?>(.2)*\dir{>}?>(.99)*\dir{>}?>(.55)*\dir{>}
,(0,3);(0,-3)**\crv{(0,4)&(0,5)&(-4,6)&(-6,0)&(-4,-6)&(0,-5)&(0,-4)}
?>(.55)*\dir{<}
,(2,-11)*{\scriptstyle{1}}
,(2,11)*{\scriptstyle{1}}
,(2,0)*{\scriptstyle{2}}
,(-6,0)*{\scriptstyle{1}}
\endxy}\{-1\}
\ar[rr]^-{
 \chi_0\otimes 1}
\ar[dd]_{1\otimes \mu}^{\wr}
&&
{\xy
(-10,-10);(-10,10)**\crv{(-6,-6)&(-5,-5)&(-5,-2)&(-5,-2)&(-5,5)&(-6,6)}
?>(.05)*\dir{>}?>(.99)*\dir{>}?>(.55)*\dir{>}
,(10,-10);(10,10)**\crv{(6,-6)&(5,-5)&(5,-2)&(5,2)&(5,5)&(6,6)}
?>(.03)*\dir{<}?>(.92)*\dir{<}?>(.45)*\dir{<}
,(-5,0);(5,0)**\crv{(-5,-2)&(-5,-3)&(-4,-5)&(-2,-6)&(2,-6)&(4,-5)&(5,-3)&(5,-2)}?>(.5)*\dir{<}
,(-5,0);(5,0)**\crv{(-5,2)&(-5,3)&(-4,5)&(-2,6)&(2,6)&(4,5)&(5,3)&(5,2)}?>(.5)*\dir{>}
,(-11,-10)*{\scriptstyle{1}}
,(-11,10)*{\scriptstyle{1}}
,(11,-10)*{\scriptstyle{1}}
,(11,10)*{\scriptstyle{1}}
,(-7,0)*{\scriptstyle{2}}
,(0,-8)*{\scriptstyle{1}}
,(6.5,0)*{\scriptstyle{2}}
,(0,8)*{\scriptstyle{1}}
\endxy}
\ar[dd]_{\nu}^{\wr}
\\
\\
[n-1]{\xy
(-5,-10);(-5,10)**\dir{-}?>(.5)*\dir{>}
,(-7,-11)*{\scriptstyle{1}}
,(-7,11)*{\scriptstyle{1}}
,(5,-10);(5,10)**\dir{-}?>(.5)*\dir{<}
,(7,-11)*{\scriptstyle{1}}
,(7,11)*{\scriptstyle{1}}
\endxy}\{-1\}\ar@{=}[dd]
&&
\xy
(-5,-10);(5,-10)**\crv{(-5,-2)&(0,-2)&(5,-2)}
?>(1)*\dir{>},
(-6,-10)*{\scriptstyle{1}},(6,-10)*{\scriptstyle{1}}
,(-5,10);(5,10)**\crv{(-5,2)&(0,2)&(5,2)}
?>(0)*\dir{<},
(-6,10)*{\scriptstyle{1}},(6,10)*{\scriptstyle{1}}
\endxy
\oplus\, [n-2]
{\xy
(-5,-10);(-5,10)**\dir{-}?>(.5)*\dir{>}
,(5,-10);(5,10)**\dir{-}?>(.5)*\dir{<}
,(-7,-10)*{\scriptstyle{1}}
,(7,-10)*{\scriptstyle{1}}
,(-7,10)*{\scriptstyle{1}}
,(7,10)*{\scriptstyle{1}}
\endxy}
\ar@{=}[dd]
\\ \\
(\mathbb{Q}\{1-n\}\oplus[n-2]){\xy
(-5,-10);(-5,10)**\dir{-}?>(.5)*\dir{>}
,(-7,-11)*{\scriptstyle{1}}
,(-7,11)*{\scriptstyle{1}}
,(5,-10);(5,10)**\dir{-}?>(.5)*\dir{<}
,(7,-11)*{\scriptstyle{1}}
,(7,11)*{\scriptstyle{1}}
\endxy}
\ar[rr]^-{\left(
\begin{smallmatrix}
0 & 0\\
0 &1 \\
\end{smallmatrix}\right)}
&&
\xy
(-5,-10);(5,-10)**\crv{(-5,-2)&(0,-2)&(5,-2)}
?>(1)*\dir{>},
(-6,-10)*{\scriptstyle{1}},(6,-10)*{\scriptstyle{1}}
,(-5,10);(5,10)**\crv{(-5,2)&(0,2)&(5,2)}
?>(0)*\dir{<},
(-6,10)*{\scriptstyle{1}},(6,10)*{\scriptstyle{1}}
\endxy
\oplus\,  [n-2]
{\xy
(-5,-10);(-5,10)**\dir{-}?>(.5)*\dir{>}
,(5,-10);(5,10)**\dir{-}?>(.5)*\dir{<}
,(-7,-10)*{\scriptstyle{1}}
,(7,-10)*{\scriptstyle{1}}
,(-7,10)*{\scriptstyle{1}}
,(7,10)*{\scriptstyle{1}}
\endxy}
}
\end{equation}

\item 
\begin{equation}
\label{eq:moynu2}\tag{moy$\nu$2}
\xymatrix{
{\xy
(-10,-10);(-10,10)**\crv{(-6,-6)&(-5,-5)&(-5,-2)&(-5,-2)&(-5,5)&(-6,6)}
?>(.05)*\dir{>}?>(.99)*\dir{>}?>(.55)*\dir{>}
,(10,-10);(10,10)**\crv{(6,-6)&(5,-5)&(5,-2)&(5,2)&(5,5)&(6,6)}
?>(.03)*\dir{<}?>(.92)*\dir{<}?>(.45)*\dir{<}
,(-5,0);(5,0)**\crv{(-5,-2)&(-5,-3)&(-4,-5)&(-2,-6)&(2,-6)&(4,-5)&(5,-3)&(5,-2)}?>(.5)*\dir{<}
,(-5,0);(5,0)**\crv{(-5,2)&(-5,3)&(-4,5)&(-2,6)&(2,6)&(4,5)&(5,3)&(5,2)}?>(.5)*\dir{>}
,(-11,-10)*{\scriptstyle{1}}
,(-11,10)*{\scriptstyle{1}}
,(11,-10)*{\scriptstyle{1}}
,(11,10)*{\scriptstyle{1}}
,(-7,0)*{\scriptstyle{2}}
,(0,-8)*{\scriptstyle{1}}
,(6.5,0)*{\scriptstyle{2}}
,(0,8)*{\scriptstyle{1}}
\endxy}
\ar[rr]^-{
1\otimes  \chi_1}
\ar[dd]_{\nu}^{\wr}
&&
{\xy
(0,-10);(0,10)**\crv{(0,-7)&(0,-6)&(0,-5)&(0,0)&(0,5)&(0,6)&(0,7)}
?>(.2)*\dir{>}?>(.99)*\dir{>}?>(.55)*\dir{>}
,(0,-3);(0,3)**\crv{(0,-4)&(0,-5)&(4,-6)&(6,0)&(4,6)&(0,5)&(0,4)}
?>(.55)*\dir{<}
,(2,-11)*{\scriptstyle{1}}
,(2,11)*{\scriptstyle{1}}
,(-2,0)*{\scriptstyle{2}}
,(7,0)*{\scriptstyle{1}}
,(10,-10);(10,10)**\dir{-}?>(.5)*\dir{<}
,(12,-11)*{\scriptstyle{1}}
,(12,11)*{\scriptstyle{1}}
\endxy}\{1\}
\ar[dd]_{\mu\otimes 1}^{\wr}
\\
\\
\xy
(-5,-10);(5,-10)**\crv{(-5,-2)&(0,-2)&(5,-2)}
?>(1)*\dir{>},
(-6,-10)*{\scriptstyle{1}},(6,-10)*{\scriptstyle{1}}
,(-5,10);(5,10)**\crv{(-5,2)&(0,2)&(5,2)}
?>(0)*\dir{<},
(-6,10)*{\scriptstyle{1}},(6,10)*{\scriptstyle{1}}
\endxy
\oplus\, [n-2]
{\xy
(-5,-10);(-5,10)**\dir{-}?>(.5)*\dir{>}
,(5,-10);(5,10)**\dir{-}?>(.5)*\dir{<}
,(-7,-10)*{\scriptstyle{1}}
,(7,-10)*{\scriptstyle{1}}
,(-7,10)*{\scriptstyle{1}}
,(7,10)*{\scriptstyle{1}}
\endxy}
\ar@{=}[dd]
&&
[n-1]{\xy
(-5,-10);(-5,10)**\dir{-}?>(.5)*\dir{>}
,(-7,-11)*{\scriptstyle{1}}
,(-7,11)*{\scriptstyle{1}}
,(5,-10);(5,10)**\dir{-}?>(.5)*\dir{<}
,(7,-11)*{\scriptstyle{1}}
,(7,11)*{\scriptstyle{1}}
\endxy}\{1\}\ar@{=}[dd]
\\ \\
\xy
(-5,-10);(5,-10)**\crv{(-5,-2)&(0,-2)&(5,-2)}
?>(1)*\dir{>},
(-6,-10)*{\scriptstyle{1}},(6,-10)*{\scriptstyle{1}}
,(-5,10);(5,10)**\crv{(-5,2)&(0,2)&(5,2)}
?>(0)*\dir{<},
(-6,10)*{\scriptstyle{1}},(6,10)*{\scriptstyle{1}}
\endxy
\oplus\,  [n-2]
{\xy
(-5,-10);(-5,10)**\dir{-}?>(.5)*\dir{>}
,(5,-10);(5,10)**\dir{-}?>(.5)*\dir{<}
,(-7,-10)*{\scriptstyle{1}}
,(7,-10)*{\scriptstyle{1}}
,(-7,10)*{\scriptstyle{1}}
,(7,10)*{\scriptstyle{1}}
\endxy}
\ar[rr]^-{\left(
\begin{smallmatrix}
0 & 1\\
0 &0 \\
\end{smallmatrix}\right)}
&&
([n-2]\oplus \mathbb{Q}\{n-1\}){\xy
(-5,-10);(-5,10)**\dir{-}?>(.5)*\dir{>}
,(-7,-11)*{\scriptstyle{1}}
,(-7,11)*{\scriptstyle{1}}
,(5,-10);(5,10)**\dir{-}?>(.5)*\dir{<}
,(7,-11)*{\scriptstyle{1}}
,(7,11)*{\scriptstyle{1}}
\endxy}
}
\end{equation}

\item 
\begin{equation}\tag{{moy$\eta$}}\label{moyeta}
\xymatrix{
\zorro \{-1\}\ar[r]^{\chi_0\otimes 1} \ar@/_2pc/[rd]_-{\tiny{\left(\begin{matrix}\chi_1&0\end{matrix}\right)}}
&\coso\ar[d]_-{\wr}^-\eta\\
& \esse\,\,\idoner\,\oplus\ti
}
\end{equation}

\end{enumerate}
\end{itemize}

\subsection{Formal complexes of MOY graphs}
The next step in our construction of the category $\mathbf{KR}$ is the definition of formal complexes of (formal direct sums of shifted) MOY-type graphs and of a set of equivalence relations on them.

\begin{definition}\label{eq:formal-complex}
A bounded formal complex of (formal direct sums of shifted) MOY-type graphs is a sequence
\[
\cdots \xrightarrow{A_{i-1}} M_i\xrightarrow{A_i}M_{i+1}\xrightarrow{A_{i+1}}\cdots
\]
with $i\in \mathbb{Z}$, where
\begin{itemize}
\item the $M_i$ are formal direct sums of shifthed MOY-type graphs, all with the same input and the same output sequences of $\uparrow^1$'s and $\downarrow^1$'s;
\item $A_i\colon M_i\to M_{i+1}$ is an element in the quotient $\widetilde{\mathcal{M}}$ of 
$\mathcal{M}$  
by the ideal generated by the relations in Section \ref{sec:relations};
\item one has $A_{i+1}\circ A_i=0$, for any $i\in \mathbb{Z}$; 
\item there exist two integers $i_{\min}$ and $i_{\max}$ such that $M_i=\mathbf{0}$ for any $i<i_{\min}$ ad any $i>i_{\max}$.
\end{itemize}
We say that the component $M_i$ of $M_{\bullet}$ is in degree $i$.\par
A morphism $\varphi_\bullet\colon M_\bullet\to N_\bullet$ of formal complexes of (formal direct sums of shifthed) MOY-type graphs is a sequence of morphisms $\varphi_i\colon M_i\to N_i$ in $\widetilde{\mathcal{M}}$ such that the diagram
\[
\xymatrix{
\cdots\ar[r]& M_i\ar[d]_{\varphi_i}\ar[r]^{A_i}&M_{i+1}\ar[r]\ar[d]^{\varphi_{i+1}}&\cdots\\
\cdots\ar[r]& N_i\ar[r]^{B_i}&N_{i+1}\ar[r]&\cdots
}
\]
commutes (i.e., $\varphi_{i+1}\circ A_i=B_i\circ\varphi_i$ in $\widetilde{\mathcal{M}}$, for any $i\in \mathbb{Z}$).
\end{definition}

\begin{remark}
The equation $A_{i+1}\circ A_i=0$ in Definition \ref{eq:formal-complex} makes sense as $\mathcal{M}$ modulo the ideal generated by the relations in Section \ref{sec:relations} is a $\mathbb{Q}$-vector space.
\end{remark}

\begin{remark}
As all the MOY-type graphs appearing in a formal complex have the same input and the same output sequences of $\uparrow^1$'s and $\downarrow^1$'s it is meaningful to talk of the source sequence and of the target sequence of a formal complex of  (formal direct sums of shifted) MOY-type graphs. 
\end{remark}

\begin{definition}
The \emph{tensor product} of two formal complexes of (formal direct sums of shifted) MOY-type graphs $M_\bullet$ and $N\bullet$ is the formal complex
\[
(M\otimes N)_\bullet := \mathrm{tot}\left(M_\bullet\otimes N_\bullet\right),
\]
where $M_\bullet\otimes N_\bullet$ is the formal bicomplex
\[
\xymatrix{
&\vdots&&\vdots\\
\cdots\ar[r]&M_i\otimes N_{j+1} \ar[u]\ar[rr]^{A_{i+1}\otimes 1_{N_{j+1}}} && M_{i+1}\otimes N_{j+1} \ar[u]\ar[r]&\cdots\\
\cdots\ar[r]&M_i\otimes N_j \ar[rr]^{A_i\otimes 1_{N_j}}\ar[u]^{1_{M_{i}}\otimes B_j}&& M_{i+1}\otimes N_j\ar[u]_{1_{M_{i+1}}\otimes B_j} \ar[r]&\cdots\\
&\vdots\ar[u]&&\vdots\ar[u]
}
\]
\end{definition}
If $M_\bullet$ and $N\bullet$ are two formal complexes of (formal direct sums of shifted) MOY-type graphs and the output sequence of $\uparrow^1$'s and $\downarrow^1$'s of $N_\bullet$ coincides with the input sequence of $M_\bullet$, then all the compositions $M_i\circ N_j$ are defined and one can give the following definition.
\begin{definition}
The \emph{composition} of two formal complexes of (formal direct sums of shifted) MOY-type graphs $M_\bullet$ and $N\bullet$ such that the output sequence of $\uparrow^1$'s and $\downarrow^1$'s of $N_\bullet$ coincides with the input sequence of $M_\bullet$ is the formal complex
\[
(M\circ N)_\bullet := \mathrm{tot}\left(M_\bullet\circ N_\bullet\right),
\]
where $M_\bullet \circ\ N_\bullet$ is the formal bicomplex
\[
\xymatrix{
&\vdots&&\vdots\\
\cdots\ar[r]&M_i\circ N_{j+1} \ar[u]\ar[rr]^{A_{i+1}\circ 1_{N_{j+1}}} && M_{i+1}\circ N_{j+1} \ar[u]\ar[r]&\cdots\\
\cdots\ar[r]&M_i\circ N_j \ar[rr]^{A_i\circ 1_{N_j}}\ar[u]^{1_{M_{i}}\circ B_j}&& M_{i+1}\circ N_j\ar[u]_{1_{M_{i+1}}\circ B_j} \ar[r]&\cdots\\
&\vdots\ar[u]&&\vdots\ar[u]
}
\]
\end{definition}

\vskip .7 cm
We introduce on (bounded) formal complexes a sum operation given by direct sum of complexes,
\[
M_\bullet\oplus N_\bullet= \left(
\cdots \xrightarrow{A_{i-1}\oplus B_{i-1}} M_i\oplus N_i\xrightarrow{A_i\oplus B_i}M_{i+1}\oplus N_{i+1}\xrightarrow{A_{i+1}\oplus B_{i+1}}\cdots\right)
\]
 as well as the following equivalence relations:
\begin{itemize}
\item (eq1) Isomorphic formal complexes are equivalent. That is, if we have a morphism of formal complexes
\[
\xymatrix{
\cdots\ar[r]& M_i\ar[d]_{\varphi_i}^{\wr}\ar[r]^{A_i}&M_{i+1}\ar[r]\ar[d]^{\varphi_{i+1}}_{\wr}&\cdots\\
\cdots\ar[r]& N_i\ar[r]^{B_i}&N_{i+1}\ar[r]&\cdots
}
\]
where all the $\varphi_i$'s are invertible (i.e., for each $i$ there exists an element $\psi_i\colon N_i\to M_i$ in $\widetilde{\mathcal{M}}$ such that $\varphi_i\circ\psi_i=1_{N_i}$ and $\psi_i\circ\varphi_i=1_{M_i}$ in $\widetilde{\mathcal{M}}$), then
\[
\cdots\to M_i\xrightarrow{A_i} M_{i+1}\to\cdots
\]
is equivalent to
\[
\cdots\to N_i\xrightarrow{B_i} N_{i+1}\to\cdots.
\]
\item (eq2)  A formal complex of the form
\[
\cdots \to M_{i-1}\xrightarrow{\left(
\begin{matrix}
\alpha_{i-1}\\A_{i-1}
\end{matrix}
\right)}
N\oplus M_i\xrightarrow{\left(
\begin{matrix}
1&\alpha_i\\0&A_i
\end{matrix}
\right)}
N\oplus M_{i+1}\xrightarrow{\left(
\begin{matrix}
0&A_{i+1}
\end{matrix}
\right)}
 M_{i+2}\to\cdots
\]
is equivalent to the formal complex
\[
\cdots\to  M_{i-1}\ \xrightarrow{A_{i-1}}
M_i\xrightarrow{A_i} M_{i+1}\xrightarrow{A_{i+1}} M_{i+2}
\to \cdots
\]
That is, we can ``simplify'' terms $\Gamma\xrightarrow{1}\Gamma$ if they appear as subcomplexes.

\item (eq3) A formal complex of the form
\[
\cdots \to M_{i-1}\xrightarrow{\left(
\begin{matrix}
A_{i-1}\\0
\end{matrix}
\right)}
 M_i \oplus N\xrightarrow{\left(
\begin{matrix}
A_i&\alpha_i\\0&1
\end{matrix}
\right)}
M_{i+1}\oplus N\xrightarrow{\left(
\begin{matrix}
A_{i+1}&\alpha_{i+1}
\end{matrix}
\right)}
 M_{i+2}\to\cdots
\]
is equivalent
to the formal complex
\[
\cdots\to  M_{i-1}\ \xrightarrow{A_{i-1}}
M_i\xrightarrow{A_i} M_{i+1}\xrightarrow{A_{i+1}} M_{i+2}
\to \cdots
\]
That is, we can ``simplify'' terms $\Gamma\xrightarrow{1}\Gamma$ if they appear as  quotient complexes.
\end{itemize}
\begin{notation}
We denote by $\mathbf{FC_{MOY}}$ the set of equivalence classes of formal complexes of (formal direct sums of shifted) MOY-type graphs, with respect to the equivalence relation generated by the equivalence relations (eq1)-(eq3) above. The equivalence class of the formal complex $M_\bullet$ will be denoted by the symbol $[M_\bullet]$.
\end{notation}
By the usual properties of direct sums and tensor product of complexes and totalization functors one immediately sees that the tensor product, the composition and the direct sum operations are compatible with the equivalence relations (eq1)-(eq3) and so induce corresponding operations on $\mathbf{FC_{MOY}}$. More formally, this is stated as follows.
\begin{proposition}\label{prop:operations}
The operations
\[
+,\otimes,\circ \colon  \mathbf{FC_{MOY}}\times \mathbf{FC_{MOY}} \to \mathbf{FC_{MOY}}
\]
defined by
\begin{align*}
[M_\bullet]+[N_\bullet]&=[M_\bullet\oplus N_\bullet]\\
[M_\bullet]\otimes[N_\bullet]&=[M_\bullet\otimes N_\bullet]\\
[M_\bullet]\circ[N_\bullet]&=[M_\bullet\circ N_\bullet]
\end{align*}
are well defined and satisfy all the expected associativity, commutativity and distributivity conditions. In particular $(\mathbf{FC_{MOY}},+,\otimes)$ is an abelian semiring.
\end{proposition}
\begin{remark}
The operation $([M_\bullet],[N_\bullet])\mapsto [M_\bullet]\circ[N_\bullet]$ in the above proposition is defined when the common target sequence of $\uparrow^1$'s and $\downarrow^1$'s of the MOY-type graphs appearing in $N_\bullet$ coincides with the source sequence MOY-type graphs appearing in $M_\bullet$. These sequences are well defined, i.e., they do not depend on the particular representatives $M_\bullet$ and $N_\bullet$ chosen for the equivalence classes $[M_\bullet]$ and $[N_\bullet]$.
\end{remark}
\begin{remark}
The abelian semigroup $(\mathbf{FC_{MOY}},+)$ carries a natural action of the semiring $\mathbb{Z}_{\geq 0}[q,q^{-1}]$ of Laurent polynomials in the variable $q$ with coefficients in the abelian semiring $\mathbb{Z}_{\geq 0}$ of nonnegative integers given by
\[
q^i\cdot [M_\bullet]=[M_\bullet\{i\}],
\]
extended by additivity. This makes $(\mathbf{FC_{MOY}},+)$ a \emph{semimodule} for the semiring $\mathbb{Z}_{\geq 0}[q,q^{-1}]$. Notice that the $\mathbb{Z}_{\geq 0}[q,q^{-1}]$-action preserves both the source and the target sequences of elements in  $\mathbf{FC_{MOY}}$.
\end{remark}

\subsection{The category $\mathbf{KR}$}
We are finally in position to define the category $\mathbf{KR}$ announced at the beginning of this section.
\begin{notation}
For any two sequences $\vec{i}$ and $\vec{j}$ of $\uparrow^1$'s and $\downarrow^1$'s, we write $\mathbf{FC_{MOY}}(\vec{i},\vec{j})$ for the subset of $\mathbf{FC_{MOY}}$ consisting of all equivalence classes of formal complexes of (formal direct sums of shifted) MOY-type graphs with input sequence $\vec{i}$ and output sequence $\vec{j}$.
\end{notation}

\begin{definition}
The monoidal category $\mathbf{KR}$  is defined as follows:
\begin{itemize}
\item objects in $\mathbf{KR}$ are finite sequences of $\uparrow^1$'s and $\downarrow^1$'s;
\item the set $\mathbf{KR}(\vec{i},\vec{j})$ of morphisms between two objects $\vec{i}$ and $\vec{j}$ is the $\mathbb{Z}_{\geq 0}[q,q^{-1}]$-semimodule $\mathbf{FC_{MOY}}(\vec{i},\vec{j})$.
\end{itemize}
Composition of morphisms is given by the composition
\[
\mathbf{FC_{MOY}}(\vec{j},\vec{k})\times \mathbf{FC_{MOY}}(\vec{i},\vec{j})\to \mathbf{FC_{MOY}}(\vec{i},\vec{k})
\]
of Proposition \ref{prop:operations}.\par
The identity morphism is given by the (equivalence class of the) formal complex consisting of the MOY-type diagram with just straight vertical lines connecting the bottom boundary with the top boundary in degree zero, and $\mathbf{0}$ in all the others degrees.\par
The tensor product of objects is given by concatenation. Tensor product of morphisms is given by the tensor product
\[
\mathbf{FC_{MOY}}(\vec{i}_1,\vec{j}_1)\times \mathbf{FC_{MOY}}(\vec{j}_1,\vec{j}_2)\to \mathbf{FC_{MOY}}((\vec{i}_1\otimes \vec{i}_2,\vec{j}_1\otimes \vec{j}_2)
\]
of Proposition \ref{prop:operations}. The unit object for the tensor product is the empty sequence.
\end{definition}
\begin{example}
The identity of $\uparrow^1\,\, \downarrow^1 \,\,\uparrow^1\,\, \uparrow^1$ is
\[
\left[
\cdots\to \mathbf{0}\to\mathbf{0}\to
\xy
(-6,-5);(-6,5)**\dir{-}
?>(1)*\dir{>},
(-2,-5);(-2,5)**\dir{-}
?>(0)*\dir{<},
(2,-5);(2,5)**\dir{-}
?>(1)*\dir{>},
(6,-5);(6,5)**\dir{-}
?>(1)*\dir{>},
(-4.7,-5)*{\scriptstyle{1}},
(-.7,-5)*{\scriptstyle{1}},
(3.3,-5)*{\scriptstyle{1}},
(7.3,-5)*{\scriptstyle{1}}
\endxy\,
\to \mathbf{0}\to\mathbf{0}\to\cdots\right]
\]
with $\xy
(-6,-5);(-6,5)**\dir{-}
?>(1)*\dir{>},
(-2,-5);(-2,5)**\dir{-}
?>(0)*\dir{<},
(2,-5);(2,5)**\dir{-}
?>(1)*\dir{>},
(6,-5);(6,5)**\dir{-}
?>(1)*\dir{>},
(-4.7,-5)*{\scriptstyle{1}},
(-.7,-5)*{\scriptstyle{1}},
(3.3,-5)*{\scriptstyle{1}},
(7.3,-5)*{\scriptstyle{1}}
\endxy$ in degree zero.
\end{example}
Next, we introduce on $\mathbf{KR}$ the structure of rigid category. As $\mathbf{KR}$ is monoidally generated by $\uparrow^1$ and $\downarrow^1$ it will suffice defining evaluation and coevaluation morphisms for these.
\begin{proposition}
The evaluation and coevaluation morphisms
\[
\mathrm{ev}_{\uparrow^1}=\left[
\cdots\to \mathbf{0}\to\mathbf{0}\to
\evop
\to \mathbf{0}\to\mathbf{0}\to\cdots\right]
\]
\[
\mathrm{ev}_{\downarrow^1}=\left[
\cdots\to \mathbf{0}\to\mathbf{0}\to
\ev
\to \mathbf{0}\to\mathbf{0}\to\cdots\right]
\]
\[
\mathrm{coev}_{\uparrow^1}=\left[
\cdots\to \mathbf{0}\to\mathbf{0}\to
\coev
\to \mathbf{0}\to\mathbf{0}\to\cdots\right]
\]
\[
\mathrm{coev}_{\downarrow^1}=\left[
\cdots\to \mathbf{0}\to\mathbf{0}\to
\coevop
\to \mathbf{0}\to\mathbf{0}\to\cdots\right],
\]
where the nonzero graph is in degree zero, endow $\mathbf{KR}$ with the structure of balanced rigid category and exhibit $\uparrow^1$ as the (left and right) dual of $\downarrow^1$ and vive versa.  
\end{proposition}
\begin{proof}
As the formal complexes defining the evaluation and coevaluation morphisms in $\mathbf{KR}$ are concentarted in degree zero, the verification of the Zorro moves is immediate. 
\end{proof}

\subsection{The braidings in $\mathbf{KR}$} 
Abstracting from \cite{Khovanov-Rozansky} we now define two morphisms
\[
\sigma^+,\sigma^-\colon \uparrow^1\,\,\uparrow^1\,\, \to\,\, \uparrow^1\,\,\uparrow^1
\]
in $\mathbf{KR}$ which satisfy Reidemeister moves invariance. As $\mathbf{KR}$ is monoidally generated by $\uparrow^1$ and its dual $\downarrow^1$, this will endow $\mathbf{KR}$ with the structure of untwisted braided category. This is the content of the following
\begin{proposition}
The morphisms
\[
\sigma^+=\left[
\cdots\to \mathbf{0}\to\mathbf{0}\to
\idtwob\{n-1\}\xrightarrow{\chi_0} \esseb\{n\}
\to \mathbf{0}\to\mathbf{0}\to\cdots\right]
\]
with $\idtwob\{n-1\}$ in degree zero, and
\[
\sigma^-=\left[
\cdots\to \mathbf{0}\to\mathbf{0}\to
\esseb\{-n\}\xrightarrow{\chi_1} \idtwob\{1-n\}
\to \mathbf{0}\to\mathbf{0}\to\cdots\right]
\]
with $\idtwob\{1-n\}$ in degree zero, endow $\mathbf{KR}$ with the structure of untwisted braided category, and so there is a distinguished braided monoidal functor
\[
Kh_n\colon \mathbf{TD}\to \mathbf{KR}
\]
from the category $\mathbf{TD}$ of tangle diagrams, mapping the object $\uparrow$ of $\mathbf{TD}$ to the object $\uparrow^1$ of $\mathbf{KR}$.
\end{proposition}
\begin{proof}
Since the relations on morphisms we have imposed encode in particular the properties of complexes of matrix factorizations used by Khovanov and Rozansky in their proof of the invariance of level $n$ Khovanov homology by Reidemeister moves,
the proof of the Proposition is verbatim adapted from the original proof in \cite{Khovanov-Rozansky}, so we omit it. A detailed derivation can be found in \cite{Omid}.
\end{proof}

As explained in the Introduction, the functor $Kh_n$ determines the Khovanov-Rozansky polynomial for a link $\Gamma$ such that $Kh_n(\Gamma)=[V_\bullet]\otimes \emptyset$
where   
\[
V_\bullet=\left(\cdots\to V_{i-1}\xrightarrow{0} V_i\xrightarrow{0} V_{i+1}\cdots \right),
\]
is some complex of graded $\mathbb{Q}$-vector spaces with trivial differentials. This fact will be used in the following Section to compute the Khovanov-Rozansky polynomials of 2-strand braid links.

\section{The $Kh_n$ invariant of 2-strand braid links}\label{chapter:computation}
We conclude this article by computing the Khovanov homology of the 2-strand braid link with $k$ crosssings, i.e., of the link
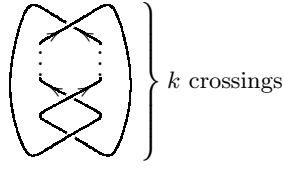
\[
\trsplusk{k}
\]
This is connected, i.e., it is a knot, for odd $k$, and has exactly two components for $k$ even. The formula we are going to find for the Poincar\'e polynomial of the Khovanov homology of the 2-strand braid with $k$ crossings, in the odd case will extend to $n\leq 4$ the results by Rasmussen \cite{Rasmussen}, and will agree  with the prediction by Morozov and Anokhina \cite{Morozov1}. For an even number of crossings, the formula we will find will agrees with the prediction by Morozov and Anokhina \cite{Morozov1}, and with the string theory inspired prediction by Gukov, Iqbal, Koz\c{c}az and Vafa for the particular case of two crossings (the Hopf link) \cite{Gukov}. For $n=2$ and arbitrary $k$, the formula agrees with results by Beliakova--Putyra--Wehrli and Nizami--Munir--Sohail--Usman  \cite{bpw19,nmsu16}. Finally, for low values of $n$ and $k$, these formulas has been computer checked by Carqueville and Murfet in \cite{Carqueville-Murfet}. More precisely, the case of 2 crossings has been computer checked for $n\leq 11$,  the case of 3 crossings has been computer checked for $n\leq 11$, the case of 4 crossings for $n\leq 5$, the case of 5 crossings for $n\leq 6$, and the case of 6 crossings for $n\leq4$. 

\subsection{Preliminary lemmas}
Here we provide a few preliminary results to be used in the computation of the Khovanov homology of the 2-strand braid links. In what follows, ``complex'' (resp.  ``bicomplex'') always mean ``formal complex  (resp., bicomplex) of formal sums of shifted MOY-type graphs''. 

\begin{lemma}\label{lemma:one}
The bicomplex
\[
\xymatrix{
\idtwob\ar[r]^{\chi_0}\ar[d]_{\chi_0}&\esseb\{1\}\ar[d]^{\chi_0\circ 1}\\
\esseb\{1\}\ar[r]^-{1\circ\chi_0}&\essecircesseb\{2\}
}
\]
is equivalent to the bicomplex
\[
\xymatrix{
\idtwob\ar[r]^{\chi_0}\ar[d]_{\chi_0}& \esseb\{1\}\ar[d]^{\tiny{\left(\begin{matrix}1\\ \alpha\end{matrix}\right)}}\\
\esseb\{1\}\ar[r]^-{\tiny{\left(\begin{matrix}1\\ 0\end{matrix}\right)}}&\esseb\{1\}\oplus \esseb\{3\}
}
\]
\end{lemma}
\begin{proof}
In the diagram
\[
\xymatrix{
\idtwob\ar[d]_{\chi_0}\ar[r]^{\chi_0}&\esseb\{1\}\ar[d]^{\chi_0\circ 1}\ar@/^1pc/[rrdd]^{\tiny{\left(\begin{matrix}1\\ \alpha\end{matrix}\right)}}\\
\esseb\{1\}\ar[r]^-{1\circ\chi_0}\ar@/_3pc/[rrrd]_{\tiny{\left(\begin{matrix}1\\ 0\end{matrix}\right)}}&\essecircesseb\{2\}\ar[drr]^{\varphi}_{\sim}\\
&&&\esseb\{1\} \oplus \esseb\{3\}
}
\]
the morphism $\varphi$ is an isomorphism and all the subdiagrams commute, due to relations (\ref{eq:moy2a}) and (\ref{eq:moy2b}).
\end{proof}

\begin{lemma}\label{lemma:two}
The bicomplex
\[
\xymatrix{
\esseb\ar[r]^{\alpha}\ar[d]_{\chi_0\circ 1}& \esseb\{2\}\ar[d]^{\chi_0\circ 1}\\
\essecircesseb\{1\}\ar[r]_-{1\circ \alpha}&\essecircesseb\{3\}
}
\]
is equivalent to the bicomplex
\[
\xymatrix{
\esseb\ar[r]^{\alpha}\ar[d]_{\tiny{\left(\begin{matrix}1\\ \alpha\end{matrix}\right)}}& \esseb\{2\}\ar[d]^{\tiny{\left(\begin{matrix}1\\ \gamma\end{matrix}\right)}}\\
\esseb\oplus \esseb\{2\}\ar[r]_-{\tiny{\left(\begin{matrix}0  & 1\\ 0 & 0\end{matrix}\right)}}&\esseb\{2\}\oplus \esseb\{4\}
}
\]
\end{lemma}
\begin{proof}
In the diagram
\[
\xymatrix@C=.5pc{
&\esseb\ar@/_1pc/[dddd]_{\tiny{\left(\begin{matrix}1\\ \alpha\end{matrix}\right)}}
\ar[ddr]^{\chi_0\circ 1}\ar[rrrr]^{\alpha}&&&&\esseb\{2\}\ar[ddl]_{\chi_0\circ 1}\ar@/^1pc/[dddd]^{\tiny{\left(\begin{matrix}1\\ \gamma\end{matrix}\right)}}\\
\\
&&\essecircesseb\{1\}\ar[ddl]_{\varphi}^{\sim}\ar[rr]^-{1\circ\alpha}&&\essecircesseb\{3\}\ar[ddr]^{\psi}_{\sim}\\
\\
&\esseb\oplus \esseb\{2\}\ar[rrrr]_-{\tiny{\left(\begin{matrix}0  & 1\\ 0 & 0\end{matrix}\right)}}&&&&\esseb\{2\} \oplus \esseb\{4\}
}
\]
all the subdiagrams commute, thanks to relations (\ref{eq:moy2b}), (\ref{eq:moy2c}) and (\ref{eq:shift}). As $\varphi$ and $\psi$ are isomorphisms, the top commutative diagram is equivalent to the outer commutative diagram. 
\end{proof}

\begin{corollary}
We have $\gamma\alpha=0$.
\end{corollary}
\begin{proof}
By Lemma \ref{lemma:two} we have
\[
\left(\begin{matrix}
\alpha \\ \gamma\alpha
\end{matrix}
\right)=
\left(\begin{matrix}1\\ \gamma
\end{matrix}
\right)
\alpha=
\left(\begin{matrix}0 &1\\0&0
\end{matrix}
\right)
\left(\begin{matrix}1\\
\alpha
\end{matrix}
\right)=
\left(\begin{matrix}
\alpha\\0
\end{matrix}
\right).
\]
\end{proof}

\begin{lemma}
The total complex of the bicomplex 
\[
\xymatrix{
\idtwob\ar[r]^{\chi_0}\ar[d]_{\chi_0}& \esseb\{1\}\ar[d]^{\tiny{\left(\begin{matrix}1\\ \alpha\end{matrix}\right)}}\ar[r]^{\alpha}\ar[d]& \esseb\{3\}\ar[d]^{\tiny{\left(\begin{matrix}1\\ \gamma\end{matrix}\right)}}\\
\esseb\{1\}\ar[r]_-{\tiny{\left(\begin{matrix}1\\ 0\end{matrix}\right)}}&\esseb\{1\}\oplus \esseb\{3\}\ar[r]_-{\tiny{\left(\begin{matrix}0  & 1\\ 0 & 0\end{matrix}\right)}}&\esseb\{3\}\oplus \esseb\{5\}
}
\]
is equivalent to
\[
\idtwob\xrightarrow{\chi_0} \esseb\{1\}\xrightarrow{\alpha} \esseb\{3\}\xrightarrow{\gamma} \esseb\{5\}.
\]
\end{lemma}
\begin{proof}
The total complex of the bicomplex we are considering is
\begin{align*}
\idtwob&\xrightarrow{\tiny{\left(\begin{matrix}\chi_0\\ \chi_0\end{matrix}\right)}} \esseb\{1\}\oplus \esseb\{1\} \xrightarrow{\tiny{\left(\begin{matrix}1 &-1\\ 0& -\alpha \\ 0 & \alpha\end{matrix}\right)}} \esseb\{1\}\oplus \esseb\{3\}\oplus \esseb\{3\}\\
&\xrightarrow{\tiny{\left(\begin{matrix}0  & 1 & 1\\ 0 & 0 & \gamma\end{matrix}\right)}} \esseb\{3\}\oplus \esseb\{5\}
\end{align*}
By (eq2) applied around the morphism $\left(\begin{matrix}1 &-1\\ 0& -\alpha \\ 0 & \alpha\end{matrix}\right)$, this complex is equivalent to the complex
\[
\idtwob\xrightarrow{\chi_0}  \esseb\{1\} \xrightarrow{\tiny{\left(\begin{matrix} -\alpha \\  \alpha\end{matrix}\right)}}  \esseb\{3\}\oplus \esseb\{3\}
\xrightarrow{\tiny{\left(\begin{matrix}  1 & 1\\  0 & \gamma\end{matrix}\right)}} \esseb\{3\}\oplus \esseb\{5\}
\]
Next, we apply (eq2) applied around the morphism $\left(\begin{matrix}  1 & 1\\  0 & \gamma\end{matrix}\right)$ to see that this complex is 
equivalent to the complex
\[
\idtwob\xrightarrow{\chi_0}  \esseb\{1\} \xrightarrow{\alpha}  \esseb\{3\}
\xrightarrow{\gamma}  \esseb\{5\}
\]

\end{proof}

\begin{lemma}\label{lemma:three}
The bicomplex
\[
\xymatrix{
\esseb\ar[r]^{\gamma}\ar[d]_{\chi_0\circ 1}& \esseb\{2\}\ar[d]^{\chi_0\circ 1}\\
\essecircesseb\{1\}\ar[r]_-{1\circ \gamma}&\essecircesseb\{3\}
}
\]
is equivalent to the bicomplex
\[
\xymatrix{
\esseb\ar[r]^{\gamma}\ar[d]_{\tiny{\left(\begin{matrix}1\\ \gamma\end{matrix}\right)}}& \esseb\{2\}\ar[d]^{\tiny{\left(\begin{matrix}1\\ \alpha\end{matrix}\right)}}\\
\esseb\oplus \esseb\{2\}\ar[r]_-{\tiny{\left(\begin{matrix}0  & 1\\ 0 & 0\end{matrix}\right)}}&\esseb\{2\}\oplus \esseb\{4\}
}
\]
\end{lemma}
\begin{proof}
In the diagram
\[
\xymatrix@C=.5pc{
&\esseb\ar@/_1pc/[dddd]_{\tiny{\left(\begin{matrix}1\\ \gamma\end{matrix}\right)}}
\ar[ddr]^{\chi_0\circ 1}\ar[rrrr]^{\gamma}&&&&\esseb\{2\}\ar[ddl]_{\chi_0\circ 1}\ar@/^1pc/[dddd]^{\tiny{\left(\begin{matrix}1\\ \alpha\end{matrix}\right)}}\\
\\
&&\essecircesseb\{1\}\ar[ddl]_{\psi}^{\sim}\ar[rr]^-{1\circ\gamma}&&\essecircesseb\{3\}\ar[ddr]^{\varphi}_{\sim}\\
\\
&\esseb\oplus \esseb\{2\}\ar[rrrr]_-{\tiny{\left(\begin{matrix}0  & 1\\ 0 & 0\end{matrix}\right)}}&&&&\esseb\{2\} \oplus \esseb\{4\}
}
\]
all the subdiagrams commute, thanks to relations (\ref{eq:moy2b}), (\ref{eq:moy2c}) and (\ref{eq:shift2}). As $\varphi$ and $\psi$ are isomorphisms, the top commutative diagram is equivalent to the outer commutative diagram. 
\end{proof}

\begin{corollary}
We have $\alpha\gamma=0$.
\end{corollary}
\begin{proof}
By Lemma \ref{lemma:three} we have
\[
\left(\begin{matrix}
\gamma \\ \alpha\gamma
\end{matrix}
\right)=
\left(\begin{matrix}1\\ \alpha
\end{matrix}
\right)
\gamma=
\left(\begin{matrix}0 &1\\0&0
\end{matrix}
\right)
\left(\begin{matrix}1\\
\gamma
\end{matrix}
\right)=
\left(\begin{matrix}
\gamma\\0
\end{matrix}
\right).
\]
\end{proof}

\begin{corollary}
The bicomplex
\[
\xymatrix{
\idtwob\ar[r]^{\chi_0}\ar[d]_-{\chi_0\circ 1}& \esseb\{1\}\ar[d]^-{\chi_0\circ 1}\ar[r]^{\alpha}& \esseb\{3\}\ar[d]^-{\chi_0\circ 1}
\ar[r]^{\gamma}& \esseb\{5\}\ar[d]^-{\chi_0\circ 1}\\
\esseb\{1\}\ar[r]^-{1\circ\chi_0}&\essecircesseb\{2\}\ar[r]^-{1\circ \alpha}&\essecircesseb\{4\}
\ar[r]^-{1\circ\gamma}&\essecircesseb\{6\}
}
\]
is equivalent to the bicomplex
\[
{
\xymatrix{
\idtwob\ar[r]^{\chi_0}\ar[d]_{\chi_0}& \esseb\{1\}\ar[d]^{\tiny{\left(\begin{matrix}1\\ \alpha\end{matrix}\right)}}\ar[r]^{\alpha}& \esseb\{3\}\ar[d]^{\tiny{\left(\begin{matrix}1\\ \gamma\end{matrix}\right)}}\ar[r]^{\gamma}\ar[d]& \esseb\{5\}\ar[d]^{\tiny{\left(\begin{matrix}1\\ \alpha\end{matrix}\right)}}\\
\esseb\{1\}\ar[r]_-{\tiny{\left(\begin{matrix}1\\ 0\end{matrix}\right)}}&\esseb\{1\}\oplus \esseb\{3\}\ar[r]_-{\tiny{\left(\begin{matrix}0  & 1\\ 0 & 0\end{matrix}\right)}}&\esseb\{3\}\oplus \esseb\{5\}\ar[r]_-{\tiny{\left(\begin{matrix}0  & 1\\ 0 & 0\end{matrix}\right)}}&\esseb\{5\}\oplus \esseb\{7\}
}
}
\]
\end{corollary}
\begin{proof}
Immediate from Lemma \ref{lemma:one}, Lemma \ref{lemma:two}, Lemma \ref{lemma:three}, and their proofs. 
\end{proof}

\begin{lemma}
The total complex of the bicomplex 
\[
{
\xymatrix{
\idtwob\ar[r]^{\chi_0}\ar[d]_{\chi_0}& \esseb\{1\}\ar[d]^{\tiny{\left(\begin{matrix}1\\ \alpha\end{matrix}\right)}}\ar[r]^{\alpha}& \esseb\{3\}\ar[d]^{\tiny{\left(\begin{matrix}1\\ \gamma\end{matrix}\right)}}\ar[r]^{\gamma}\ar[d]& \esseb\{5\}\ar[d]^{\tiny{\left(\begin{matrix}1\\ \alpha\end{matrix}\right)}}\\
\esseb\{1\}\ar[r]_-{\tiny{\left(\begin{matrix}1\\ 0\end{matrix}\right)}}&\esseb\{1\}\oplus \esseb\{3\}\ar[r]_-{\tiny{\left(\begin{matrix}0  & 1\\ 0 & 0\end{matrix}\right)}}&\esseb\{3\}\oplus \esseb\{5\}\ar[r]_-{\tiny{\left(\begin{matrix}0  & 1\\ 0 & 0\end{matrix}\right)}}&\esseb\{5\}\oplus \esseb\{7\}
}
}
\]
is equivalent to
\[
\idtwob\xrightarrow{\chi_0} \esseb\{1\}\xrightarrow{\alpha} \esseb\{3\}\xrightarrow{\gamma} \esseb\{5\}\xrightarrow{\alpha}\esseb\{7\} .
\]
\end{lemma}

\begin{proof}
The total complex of the bicomplex we are considering is

\[
\idtwob\xrightarrow{\tiny{\left(\begin{matrix}\chi_0\\ \chi_0\end{matrix}\right)}} \esseb\{1\}\oplus \esseb\{1\} \xrightarrow{\tiny{\left(\begin{matrix}1 &-1\\ 0& -\alpha \\ 0 & \alpha\end{matrix}\right)}} \esseb\{1\}\oplus \esseb\{3\}\oplus \esseb\{3\}
\xrightarrow{\tiny{\left(\begin{matrix}0  & 1 & 1\\ 0 & 0 & \gamma\\ 0 & 0 &\gamma \end{matrix}\right)}} \esseb\{3\}\oplus \esseb\{5\}\oplus \esseb\{5\}
\xrightarrow{\tiny{\left(\begin{matrix}  0 & 1 &-1\\  0 & 0 &-\alpha\end{matrix}\right)}} \esseb\{5\}\oplus \esseb\{7\}
\] 
By (eq2) applied around the morphism $\left(\begin{matrix}1 &-1\\ 0& -\alpha \\ 0 & \alpha\end{matrix}\right)$, this complex is equivalent to the complex

\[
\idtwob\xrightarrow{\chi_0}  \esseb\{1\} \xrightarrow{\tiny{\left(\begin{matrix} -\alpha \\  \alpha\end{matrix}\right)}}  \esseb\{3\}\oplus \esseb\{3\}
\xrightarrow{\tiny{\left(\begin{matrix}  1 & 1\\  0 & \gamma \\  0 & \gamma\end{matrix}\right)}} \esseb\{3\}\oplus \esseb\{5\}\oplus \esseb\{5\}\xrightarrow{\tiny{\left(\begin{matrix}  0 & 1 &-1\\  0 & 0 &-\alpha\end{matrix}\right)}}\esseb\{5\}\oplus \esseb\{7\}
\]

Next, we apply (eq2) applied around the morphism $\left(\begin{matrix}  1 & 1\\  0 & \gamma\\  0 & \gamma\end{matrix}\right)$ to see that this complex is 
equivalent to the complex
\[
\idtwob\xrightarrow{\chi_0}  \esseb\{1\} \xrightarrow{\alpha}  \esseb\{3\}
\xrightarrow{\tiny{\left(\begin{matrix}   \gamma\\   \gamma\end{matrix}\right)}}  \esseb\{5\}\oplus \esseb\{5\}
\xrightarrow{\tiny{\left(\begin{matrix}   1 &-1\\   0 &-\alpha\end{matrix}\right)}}\esseb\{5\}\oplus \esseb\{7\}
\]
Finally, we apply (eq2) applied around the morphism $\left(\begin{matrix}  1 & -1\\  0 & -\alpha\end{matrix}\right)$ to see that this complex is 
equivalent to the complex
\[
\idtwob\xrightarrow{\chi_0}  \esseb\{1\} \xrightarrow{\alpha}  \esseb\{3\}
\xrightarrow{\gamma}  \esseb\{5\}
\xrightarrow{-\alpha}\oplus \esseb\{7\}
\]
This is in turn manifestly equivalent to
the complex
\[
\idtwob\xrightarrow{\chi_0}  \esseb\{1\} \xrightarrow{\alpha}  \esseb\{3\}
\xrightarrow{\gamma}  \esseb\{5\}
\xrightarrow{\alpha} \esseb\{7\}
\]
by (eq1) and the evident isomorphism of complexes
\[
\xymatrix{
\idtwob\ar[r]^{\chi_0}\ar[d]_{1}  & \esseb\{1\}\ar[r]^{\alpha}\ar[d]_{1} &  \esseb\{3\}\ar[r]^{\gamma}\ar[d]_{1} 
&  \esseb\{5\}\ar[r]^{-\alpha}\ar[d]_{1} &
\esseb\{7\}\ar[d]_{-1}\\
\idtwob\ar[r]^{\chi_0}  & \esseb\{1\}\ar[r]^{\alpha}& \esseb\{3\}\ar[r]^{\gamma} 
&  \esseb\{5\}\ar[r]^{\alpha}&
\esseb\{7\}
}
\]
\end{proof}

From this point on we can proceed inductively. This way one proves the following
\begin{proposition}\label{main-prop}
The total complex of the bicomplex
\[
{
\xymatrix{
\idtwob\ar[r]^{\chi_0}\ar[d]_-{\chi_0\circ 1}& \esseb\{1\}\ar[d]^-{\chi_0\circ 1}\ar[r]^{\alpha}& \esseb\{3\}\ar[d]^-{\chi_0\circ 1}
\ar[r]^{\gamma}& \esseb\{5\}\ar[d]^-{\chi_0\circ 1}\ar[r]^-{\alpha}&\cdots\ar[r]^-{\alpha}& \esseb\{4k-1\}\ar[d]^-{\chi_0\circ 1}\\
\esseb\{1\}\ar[r]^-{1\circ\chi_0}&\essecircesseb\{2\}\ar[r]^-{1\circ \alpha}&\essecircesseb\{4\}\ar[r]^-{1\circ\gamma}&\essecircesseb\{6\}
\ar[r]^-{1\circ\alpha}&\cdots\ar[r]^-{1\circ\alpha}& \essecircesseb\{4k\}
}}
\]
is equivalent to the complex
\[
{
\idtwob\xrightarrow{\chi_0} \esseb\{1\}\xrightarrow{\alpha} \esseb\{3\}\xrightarrow{\gamma} \esseb\{5\}\xrightarrow{\alpha} \cdots
\xrightarrow{\alpha}\esseb\{4k-1\}\xrightarrow{\gamma}\esseb\{4k+1\}
}.
\]
and the total complex of the bicomplex
\[
{
\xymatrix{
\idtwob\ar[r]^{\chi_0}\ar[d]_-{\chi_0\circ 1}& \esseb\{1\}\ar[d]^-{\chi_0\circ 1}\ar[r]^{\alpha}& \esseb\{3\}\ar[d]^-{\chi_0\circ 1}
\ar[r]^{\gamma}& \esseb\{5\}\ar[d]^-{\chi_0\circ 1}\ar[r]^-{\alpha}&\cdots\ar[r]^-{\gamma}& \esseb\{4k+1\}\ar[d]^-{\chi_0\circ 1}\\
\esseb\{1\}\ar[r]^-{1\circ\chi_0}&\essecircesseb\{2\}\ar[r]^-{1\circ \alpha}&\essecircesseb\{4\}\ar[r]^-{1\circ\gamma}&\essecircesseb\{6\}
\ar[r]^-{1\circ\alpha}&\cdots\ar[r]^-{1\circ\gamma}& \essecircesseb\{4k+2\}
}}
\]
is equivalent to the complex
\[
{
\idtwob\xrightarrow{\chi_0} \esseb\{1\}\xrightarrow{\alpha} \esseb\{3\}\xrightarrow{\gamma} \esseb\{5\}\xrightarrow{\alpha} \cdots
\xrightarrow{\gamma}\esseb\{4k+1\}\xrightarrow{\alpha}\esseb\{4k+3\}
}.
\]
\end{proposition}

\subsection{The computation}
Proposition \ref{main-prop} immediately leads us to our main result.

\begin{theorem}\label{main-thm}
We have, for any $k\geq 0$,
\begin{align*}
Kh_n\left(\,\splusk{2k+1}\, \right)&
=\left[\idtwob\xrightarrow{\chi_0} \esseb\{1\}\xrightarrow{\alpha} \esseb\{3\}\xrightarrow{\gamma} \esseb\{5\}\right.
\\
&\left.\xrightarrow{\alpha} \cdots
\xrightarrow{\alpha}\esseb\{4k-1\}\xrightarrow{\gamma}\esseb\{4k+1\}\right]\{(n-1)(2k+1)\}
\end{align*}
and, for any $k\geq 1$,
\begin{align*}
Kh_n\left(\,\splusk{2k}\, \right)&
=\left[\idtwob\xrightarrow{\chi_0} \esseb\{1\}\xrightarrow{\alpha} \esseb\{3\}\xrightarrow{\gamma} \esseb\{5\}\right.
\\
&\left.\xrightarrow{\alpha} \cdots
\xrightarrow{\gamma}\esseb\{4k-3\}\xrightarrow{\alpha}\esseb\{4k-1\}\right]\{(n-1)(2k)\}
\end{align*}

\end{theorem} 
\begin{proof}
We prove the statement inductively on the number $m$ of crossings. The base of the induction is $m=1$. In this case we need to show that
\[
Kh_n\left(\,\splus\, \right)
=\left[\idtwob\xrightarrow{\chi_0} \esseb\{1\}\right]\{n-1\}
\]
and this is true by definition of $Kh$. Assume now the statement has been proven up to $m=m_0$ an let us prove it for $m=m_0+1$. If $m_0=2k-1$ then we know by the inductive hypothesis that
\begin{align*}
Kh_n\left(\,\splusk{2k-1}\, \right)&
=\left[\idtwob\xrightarrow{\chi_0} \esseb\{1\}\xrightarrow{\alpha} \esseb\{3\}\xrightarrow{\gamma} \esseb\{5\}\right.
\\
&\left.\xrightarrow{\alpha} \cdots
\xrightarrow{\alpha}\esseb\{4k-5\}\xrightarrow{\gamma}\esseb\{4k-3\}\right]\{(n-1)(2k-1)\}
\end{align*}
As
\[
\splusk{2k}=\splusk{2k-1}\quad \circ \quad \splus\,,
\]
we have
\[Kh_n\left(\,\splusk{2k}\,\right)=Kh_n\left(\,\splusk{2k-1}\,\right) \circ Kh_n\left(\,\splus\,\right)=
\]
\[
=[\mathrm{tot}\left(\raisebox{24pt}{
\xymatrix{
\idtwob\ar[r]^{\chi_0}\ar[d]_-{\chi_0\circ 1}& \esseb\{1\}\ar[d]^-{\chi_0\circ 1}\ar[r]^{\alpha}& \esseb\{3\}\ar[d]^-{\chi_0\circ 1}
\ar[r]^{\gamma}& \esseb\{5\}\ar[d]^-{\chi_0\circ 1}\ar[r]^-{\alpha}&\cdots\ar[r]^-{\gamma}& \esseb\{4k-3\}\ar[d]^-{\chi_0\circ 1}\\
\esseb\{1\}\ar[r]^-{1\circ\chi_0}&\essecircesseb\{2\}\ar[r]^-{1\circ \alpha}&\essecircesseb\{4\}\ar[r]^-{1\circ\gamma}&\essecircesseb\{6\}
\ar[r]^-{1\circ\alpha}&\cdots\ar[r]^-{1\circ\gamma}& \essecircesseb\{4k-2\}
}}
\right)]\{(n-1)(2k)\}
\]
\[
={
\idtwob\xrightarrow{\chi_0} \esseb\{1\}\xrightarrow{\alpha} \esseb\{3\}\xrightarrow{\gamma} \esseb\{5\}\xrightarrow{\alpha} \cdots
\xrightarrow{\gamma}\esseb\{4k-3\}\xrightarrow{\alpha}\esseb\{4k-1\}
}]\{(n-1)(2k)\},
\]
by Proposition \ref{main-prop}. Therefore the statement is true for $m=m_0+1$ in this case. The proof for the case $m_0=2k$ is completely analogous.
\end{proof}
By ``closing the diagrams'', we immediately get the following.
\begin{corollary}
We have, for any $k\geq 0$,
\begin{align*}
Kh_n\left(\,\trsplusk{2k+1}\,\right)&
=\left[ [n] \{1-n\}\,
\xrightarrow{0} 0\xrightarrow{0} [n-1]\{4-n\} 
\xrightarrow{0} [n-1]\{4+n\} 
\xrightarrow{0} \cdots\right.\\
&\hskip -12 pt\left.\xrightarrow{0}[n-1]\{4k-n\}  
\xrightarrow{0} [n-1]\{4k+n\} \right]\{(n-1)(2k+1)\}
\end{align*}
and, for any $k\geq 1$,
\begin{align*}
Kh_n\left(\,\trsplusk{2k}\,\right)&
=\left[ [n] \{1-n\}\,
\xrightarrow{0} 0\xrightarrow{0} [n-1]\{4-n\} 
\xrightarrow{0} [n-1]\{4+n\} 
\xrightarrow{0} \cdots\right.\\
&\hskip -82 pt\left.\xrightarrow{0}[n-1]\{4(k-1)-n\}  \xrightarrow{0}[n-1]\{4(k-1)+n\}
\xrightarrow{0} [n][n-1]\{4k-1\} \right]\{(n-1)(2k)\}
\end{align*}
\end{corollary}
\begin{proof}
By Theorem \ref{main-thm} we have, for any $k\geq 0$,
\begin{align*}
Kh_n\left(\,\trsplusk{2k+1}\,\right)&
=\left[\twocircles
\xrightarrow{\chi_0} \tresse\{1\}\xrightarrow{\alpha} \tresse\{3\}\right.
\\
&\xrightarrow{\gamma} \tresse\{5\}\xrightarrow{\alpha} \cdots
\xrightarrow{\alpha}\tresse\{4k-1\}\\
&\left.\xrightarrow{\gamma}\tresse\{4k+1\}\right]\{(n-1)(2k+1)\}
\end{align*}
By equation (\ref{eq:moychi0}) we have a commutative diagram
\[
\xymatrix{
*++++++++++{\twocircles} \ar[rr]^-{\chi_0}\ar[d]_-{\lambda\otimes 1}^-{\wr} &&
*++++{{\tresse}\{1\}}\ar[d]^-{\mu\{1\}}_{\wr}\\
\iunknot\,\, \{1-n\}\oplus [n-1]\iunknot\,\,\, \{1\}\,
\ar[rr]^-{\tiny{\left(\begin{matrix}0&1\end{matrix}\right)}}&&*+++{[n-1]\iunknot\,\,\, \{1\}}
}
\]
By equation (\ref{eq:moichi0a}) we have a commutative diagram
\[
\xymatrix{
*++++{{\tresse}\{-1\}}\ar[d]^-{\mu} \ar[r]^{\alpha} &*++++{{\tresse}\{1\}}\ar[d]^-{\mu}\\
*++++{[n-1]\iunknot\,\,\{-1\}}\ar[r]^{0} &*+++{[n-1]\iunknot\,\,\{1\}}
}
\]
By equation (\ref{eq:moichi0b}) we have a commutative diagram
\[
\xymatrix{
*++++{{\tresse}\{-1\}}\ar[d]^-{\mu} \ar[r]^{\gamma} &*++++{{\tresse}\{1\}}\ar[d]^-{\mu}\\
*+++{[n-1]\iunknot\,\,\{-1\}}\ar[r]^{\epsilon} &*+++{[n-1]\iunknot\,\,\{1\}}
}
\]
Therefore, by equation (eq1) we get
\begin{align*}
Kh_n\left(\,\trsplusk{2k+1}\,\right)&
=\left[\iunknot\,\, \{1-n\}\oplus [n-1]\iunknot\,\,\, \{1\}\,
\xrightarrow{\tiny{\left(\begin{matrix}0&1\end{matrix}\right)}} {[n-1]\iunknot\,\,\{1\}}\right.\\
&\hskip -24 pt \xrightarrow{0} {[n-1]\iunknot\,\,\{3\}}
\xrightarrow{\epsilon} {[n-1]\iunknot\,\,\{5\}}\xrightarrow{0} \cdots
\xrightarrow{0}{[n-1]\iunknot\,\,\{4k-1\}}\\
&\hskip -24 pt\left.\xrightarrow{\epsilon}{[n-1]\iunknot\,\,\{4k+1\}}\right]\{(n-1)(2k+1)\}
\end{align*}
By the isomorphism
\[
\lambda\colon \iunknot\xrightarrow{\sim} [n]
\]
 and by equation (\ref{eq:moyepsilon}) we obtain
\begin{align*}
Kh_n\left(\,\trsplusk{2k+1}\,\right)&
=\left[[n] \{1-n\}\oplus [n-1][n] \{1\}\,
\xrightarrow{\tiny{\left(\begin{matrix}0&1\end{matrix}\right)}} {[n-1][n]\{1\}}\right.\\
&\hskip -120 pt \xrightarrow{0} [n-1](\mathbb{Q}\{4-n\}\oplus [n-1]\{4\}) 
\xrightarrow{\tiny{\left(\begin{matrix}0&1\\0 & 0\end{matrix}\right)}} [n-1]([n-1]\{4\}\oplus\mathbb{Q}\{4+n\}) 
\xrightarrow{0} \cdots\\
&\hskip -120 pt\left.\xrightarrow{0}[n-1](\mathbb{Q}\{4k-n\}\oplus [n-1]\{4k\}) 
\xrightarrow{\tiny{\left(\begin{matrix}0&1\\0 & 0\end{matrix}\right)}} [n-1]([n-1]\{4k\}\oplus\mathbb{Q}\{4k+n\}) \right]\{(n-1)(2k+1)\}
\end{align*}
By changing the orders in the direct sums (and so by (eq1)) we can rewrite this as
\begin{align*}
Kh_n\left(\,\trsplusk{2k+1}\,\right)&
=\left[ [n-1][n] \{1\}\oplus[n] \{1-n\}\,
\xrightarrow{\tiny{\left(\begin{matrix}1&0\end{matrix}\right)}} {[n-1][n]\{1\}}\right.\\
&\hskip -120 pt \xrightarrow{0} [n-1]([n-1]\{4\}\oplus \mathbb{Q}\{4-n\}) 
\xrightarrow{\tiny{\left(\begin{matrix}1&0\\0 & 0\end{matrix}\right)}} [n-1]([n-1]\{4\}\oplus\mathbb{Q}\{4+n\}) 
\xrightarrow{0} \cdots\\
&\hskip -120 pt\left.\xrightarrow{0}[n-1]([n-1]\{4k\}\oplus \mathbb{Q}\{4k-n\} ) 
\xrightarrow{\tiny{\left(\begin{matrix}1&0\\0 & 0\end{matrix}\right)}} [n-1]([n-1]\{4k\}\oplus\mathbb{Q}\{4k+n\}) \right]\{(n-1)(2k+1)\}
\end{align*}
By repeated use of (eq2) we therefore obtain
\begin{align*}
Kh_n\left(\,\trsplusk{2k+1}\,\right)&
=\left[ [n] \{1-n\}\,
\xrightarrow{0} 0\xrightarrow{0} [n-1]\{4-n\} 
\xrightarrow{0} [n-1]\{n+4\} 
\xrightarrow{0} \cdots\right.\\
&\hskip -12 pt\left.\xrightarrow{0}[n-1]\{4k-n\}  
\xrightarrow{0} [n-1]\{n+4k\} \right]\{(n-1)(2k+1)\}
\end{align*}

The proof for the case of an even number of crossing is perfectly analogous.
\end{proof}

\begin{remark}\label{rem:poincare}
As the MOY complex associated with an $m$-crossing 2-stand braid link can be represented by a complex with only graded vector spaces (times $\emptyset$) as objects and all zero differentials, it is completely encoded into the \emph{Poincar\'e polynomial} of the graded dimensions, obtained by the rule
\[
\left[0\to 0\to \cdots\to 0\to  \hskip -15 pt\underbrace{\mathbb{Q}\{i\}}_{\text{\tiny{horizontal degree $j$}}}  \hskip -15 pt\to 0\to\cdots\to 0\right] \mapsto q^it^j
\]
This way we obtain
\begin{align*}
Kh_n\left(\,\trsplusk{2k+1}\,\right)&
=q^{(n-1)(2k+1)}\left(q^{1-n}[n]_q+\left(t^2q^{4-n}+t^3q^{4+n}+\right.\right.\\
&\qquad\qquad\left.\left. +t^4q^{8-n}+t^5q^{8+n}+\cdots+t^{2k}q^{4k-n} +t^{2k+1}q^{4k+n}\right)[n-1]_q\right)\\
&\\
&\hskip -90 pt=q^{(n-1)(2k+1)}\left(q^{1-n}[n]_q+t^2q^4(q^{-n}+tq^n)(1+t^2q^4+(t^2q^4)^2+\cdots+(t^2q^4)^{k-1})[n-1]_q \right)\\
&\\
&\hskip -90 pt=q^{(n-1)(2k+1)}\left(q^{1-n}[n]_q+(q^{-n}+tq^n)\frac{t^2q^4-t^{2k+2}q^{4k+4}}{1-t^2q^4}[n-1]_q \right)
\end{align*}
and
\begin{align*}
Kh_n\left(\,\trsplusk{2k}\,\right)&=q^{(n-1)(2k)}\left(q^{1-n}[n]_q+\left(t^2q^{4-n}+t^3q^{4+n}+\right.\right.\\
&\hskip -50 pt\left. \left.+t^4q^{8-n}+t^5q^{8+n}+\cdots+t^{2k-2}q^{4(k-1)-n} +t^{2k-1}q^{4(k-1)+n}\right)[n-1]_q+t^{2k}q^{4k-1}[n]_q[n-1]_q\right)\\
&\\
&\hskip -100 pt=q^{(n-1)(2k+1)}\left(q^{1-n}[n]_q+t^2q^4(q^{-n}+tq^n)(1+t^2q^4+\cdots+(t^2q^4)^{k-2})[n-1]_q +t^{2k}q^{4k-1}[n]_q[n-1]_q\right)\\
&\\
&\hskip -100 pt=q^{(n-1)(2k)}\left(q^{1-n}[n]_q+(q^{-n}+tq^n)\frac{t^2q^4-t^{2k}q^{4k}}{1-t^2q^4}[n-1]_q+t^{2k}q^{4k-1}[n]_q[n-1]_q \right)
\end{align*}
where
\[
[m]_q=\dim_q [m]=q^{1-m}+q^{3-m}+\cdots+q^{m-3}+q^{m-1}=\frac{q^m-q^{-m}}{q-q^{-1}}.
\]
\end{remark}
\begin{remark}
Evaluating at $t=-1$ the above Poincar\'e polynomials one recovers the known formulas for the level $n$ Jones polynomials of the $m$-crossings 2-strand braid link, see \cite{fw}:
\[
Jones_n\left(\,\trsplusk{2k+1}\,\right)=q^{(n-1)(2k+1)}\left(q^{1-n}[n]_q+(q^{-n}-q^n)\frac{q^4-q^{4k+4}}{1-q^4}[n-1]_q \right)
\]
\[
Jones_n\left(\,\trsplusk{2k}\,\right)
=q^{(n-1)(2k)}\left(q^{1-n}[n]_q+(q^{-n}-q^n)\frac{q^4-q^{4k}}{1-q^4}[n-1]_q+q^{4k-1}[n]_q[n-1]_q \right)
\]
\end{remark}

\begin{example}[The Hopf link] The 2-crossing 2-braid link is known as the \emph{Hopf link}, as collections of pairs of $S^1$ linked this way arise in the Hopf fibration $S^3\to S^2$. Specializing the formula from Remark \ref{rem:poincare} to the case $2k=2$, we find
\begin{align*}
Kh_n\left(\,\hopf\,\right)&=q^{2(n-1)}\left(q^{1-n}[n]_q+t^{2}q^{3}[n]_q[n-1]_q \right)\\
&=q^{2(n-1)}\left(q^{1-n}[n]_q+t^{2}q^{2}[n]_q([n]_q-q^{1-n}) \right)\\
&=q^{n-1}\left(\frac{q^n-q^{-n}}{q-q^{-1}}\right)+q^{2n}\left(\frac{q^n-q^{-n}}{q-q^{-1}}\right)^2t^2-q^{n-1}\left(\frac{q^n-q^{-n}}{q-q^{-1}}\right)t^2
\end{align*}
This way one sees how our general formula from Remark \ref{rem:poincare} specializes to the formula  in \cite{Carqueville-Murfet}.
\end{example}

\begin{example}[The trefoil knot] The 3-crossing 2-braid link is the \emph{trefoil knot}. Specializing the formula from Remark \ref{rem:poincare} to the case $2k+1=3$, we find
\begin{align*}
Kh_n\left(\,\trefoil\,\right)&=q^{3(n-1)}\left(q^{1-n}[n]_q+q^{-n}(q^{-n}+tq^n)t^2q^4[n-1]_q \right)\\
&=q^{2(n-1)}\left([n]_q+[n-1]_qq^{-1}(1+tq^{2n})t^2q^4 \right)\\
&=q^{2n-2}\left(\frac{q^n-q^{-n}}{q-q^{-1}}+\frac{q^{n-1}-q^{-n+1}}{q-q^{-1}}q^{-1}(1+tq^{2n})t^2q^4 \right)
\end{align*}
This way one sees how our general formula from Remark \ref{rem:poincare} specializes to the formula  in \cite{Carqueville-Murfet} (up to the change of variable $t\leftrightarrow t^{-1}$ witnessing the mirror refelction relating the trefoil knot considered here and the thefoil knot considered in \cite{Carqueville-Murfet}).
\end{example}

\end{document}